\documentclass{amsart}
\usepackage{color}

\newcommand{\CC}{\ensuremath{\mathbb{C}}}

\newcommand{\NN}{\ensuremath{\mathbb{N}}}

\newcommand{\RR}{\ensuremath{\mathbb{R}}}

   \newtheorem{lemma}{Lemma}[section]
   \newtheorem{theorem}[lemma]{Theorem}
   \newtheorem{remark}[lemma]{Remark}
   
   \newtheorem{prop}[lemma]{Proposition}
   
   \newtheorem{definition}[lemma]{Definition}

\numberwithin{equation}{section}
\renewcommand{\phi}{\varphi}
\newcommand{\N}{{\mathbb N}}

\begin{document}

\subjclass[2000]{Primary: ;  Secondary: }

\title[Stochastic Shell Models driven by a multiplicative fBm]
{Stochastic Shell Models driven by a multiplicative fractional Brownian--motion}

\author{Hakima Bessaih}
 \address[Hakima Bessaih]{Department of  Mathematics\\
 University of Wyoming\\
Laramie 82071 USA}
\email[Hakima Bessaih]{bessaih@uwyo.edu}

\author{Mar\'{\i}a J. Garrido-Atienza}\address[Mar\'{\i}a J. Garrido-Atienza]{Dpto. Ecuaciones Diferenciales y An\'alisis Num\'erico\\
Universidad de Sevilla, Apdo. de Correos 1160, 41080-Sevilla, Spain} \email[Mar\'{\i}a J. Garrido-Atienza]{mgarrido@us.es}

\author{Bj{\"o}rn Schmalfu{\ss }}
\address[Bj{\"o}rn Schmalfu{\ss }]{Institut f\"{u}r Stochastik\\
Friedrich Schiller Universit{\"a}t Jena, Ernst Abbe Platz 2, 77043\\
Jena,
Germany
 }
\email[Bj{\"o}rn Schmalfu{\ss }]{bjoern.schmalfuss@uni-jena.de}

\begin{abstract}
We prove existence and uniqueness of the solution of a stochastic shell--model. The equation is driven by an infinite dimensional fractional Brownian--motion with Hurst--parameter $H\in (1/2,1)$, and contains a non--trivial coefficient in front of the noise which satisfies special regularity conditions.
The appearing stochastic integrals are defined in a fractional sense. First, we prove the existence and uniqueness of variational solutions to approximating equations driven by piecewise linear continuous noise, for which we are able to derive important uniform estimates in some functional spaces. Then, thanks to a compactness argument and these estimates, we prove that these variational solutions converge to a limit solution, which turns out to be the unique pathwise mild solution associated to the shell--model with fractional noise as driving process.
\end{abstract}

\maketitle

\section*{\today}

\section{Introduction}

In this paper we consider some shell--models  under the influence of a noise. Shell--models of turbulence describe the evolution of
complex Fourier-like components of a scalar velocity field $u_{n}(t)\in\CC$ and the associated wavenumbers $k_{n}$,
where the discrete index $n$ is referred as the shell--index.
The evolution of the infinite sequence $(u_n)_{n\in \NN}$ is given by
\begin{equation}\label{SHELL}
\dot u_n(t)+\nu k_n^2 u_n(t)+b_n(u(t),u(t))=g_n(t, u(t))\dot{\omega}(t),
\qquad n\in \mathbb N
\end{equation}
with the constraints $u_{-1}(t)=u_0(t)=0$ and $u_n(t) \in \mathbb C$ for $n \in \NN$. $\dot{\omega}$ gives a noise path that will be described below. Here
$\nu \ge 0$ and, in analogy with Navier-Stokes--equations, $\nu$
represents a kinematic viscosity;
$k_n=k_0 \lambda^n$ ($k_0>0$ and $\lambda>1$) and $g_n$ is a forcing term.
The exact form of $b_n(u,v)\in \CC$ varies from one model to another.
However in all various models,
it is assumed that $b_n(u,v)$ is chosen in such a way that
\begin{equation}\label{incomp}
 \Re \displaystyle\sum_{n=1}^{\infty}b_n(u,v)\bar v_n=0,
\end{equation}
where $\Re$ denotes the real part and $\overline x$
 the complex conjugate of $x$.
Equation \eqref{incomp}
implies a formal law of conservation of energy in the inviscid
($\nu=0$) and unforced form of \eqref{SHELL}.
In particular, we define the bilinear terms $b_n$  as
\[b_n(u,v)=i(a k_{n+1}\bar u_{n+1} \bar  v_{n+2}+
            b k_{n}\bar  u_{n-1} \bar  v_{n+1}-
            a k_{n-1}\bar  u_{n-1} \bar  v_{n-2}-
            b k_{n-1}\bar  u_{n-2} \bar  v_{n-1})
\]
in the GOY--model (see \cite{G,goy})
and by
\[
b_n(u,v)=-i(a k_{n+1}\bar  u_{n+1}  v_{n+2}+
            b k_{n}\bar  u_{n-1} v_{n+1}+
            a k_{n-1} u_{n-1} v_{n-2}+
            b k_{n-1} u_{n-2} v_{n-1})
\]
in the SABRA--model (see \cite{sabra}). The two parameters $a, b$ are real numbers. There are several shell--models in literature, the GOY-- and SABRA--models defined above have been introduced in  \cite{G, goy,  sabra}. The viscous version of the GOY-- and SABRA--models, well posedness,  global regularity of solutions and smooth dependence on the initial data can be found in \cite{clt}.

In recent years, shell--models of turbulence have attracted a lot of interest for their ability to capture some of the statistical properties of the three-dimensional turbulence while  presenting a structure much simpler than the Navier--Stokes--equations.
The stochastic version of the GOY--model under the influence of an additive white noise has been studied in \cite{Barsanti} where some statistical properties in terms of the invariant measure have been shown. For the same model in \cite{Bessaih-Ferrario1, Bessaih-Ferrario2}  a Gau{\ss}ian invariant measure is associated and a flow constructed.

\smallskip

In this article we consider a long term multiplicative noise allowing to model
some memory effects.  Such a noise is given by a trace-class fractional Brownian--motion in our state space with Hurst--parameter $H\in (1/2,1)$,
see below. In contrast to white noise, a fractional--Brownian motion is not a martingale, and therefore the multiplicative noise term cannot be presented by an Ito--integral. However, to deal with stochastic integration where the integrator is only H{\"o}lder--continuous with an exponent larger than $1/2$, one can use
the Young--integration, see Young \cite{You36} or the adaptation to a stochastic set up by Z{\"a}hle \cite{Zah98}.  Since the definition of these integrals is based on fractional derivatives (see Samko {\it et al.} \cite{Samko} for a general presentation), this theory is often called fractional calculus.
 An advantage of this theory is that, in contrast to the Ito--integral which is given in general by a limit in probability of Darboux--sums derived from an adapted integrand, we can define our integral pathwise which means that for any sufficiently regular
integrand and integrator the integral is well defined. Or in other words, the  exceptional sets of measure zero which appear in the classical Ito--integration
does not depends on the integrand.
Moreover, integrals can be defined for non-adapted integrands.

\smallskip

The main issue of our work is to prove existence and uniqueness of a pathwise solution of the stochastic shell--model driven by a fractional multiplicative noise. 
Applying an infinite dimensional version of the fractional integration theory we are able to present \eqref{SHELL} in a mild sense where the last term of this equation generates a fractional integral.
In particular, the properties of the nonlinear term $B$ generated by the sequence $(b_i(u,v))_{i\in\NN}$ allow to present such a solution in a mild form. Nevertheless, in a first step we replace the fractional noise path by a piecewise linear continuous approximation.  Considering \eqref{SHELL} with such a noise path, we are able to construct global and unique mild solutions. It is important to emphasize that the classical contraction method cannot be used alone since the bilinear term 
$(b_i(u,v))_{i\in\NN}$ causes to have estimates that do not close with the right norms. This is why, we have, first to construct weak solutions and get some a priori estimates. These weak solutions have to be constructed with a smoother path noise in order to define the corresponding stochastic integral. The a priori estimates combined with the estimates obtained from the mild form are then used to pass to the limit by means of a compactness argument, and the limit will turn out to be a mild solution of the original problem. The uniqueness of solutions is proved by an argument that uses the balance of suitable norms. As we mentioned before, just using the mild form in its usual norm does not allow to close the estimates, reason for which we rather again combine the a priori estimates and the norms obtained from the mild form to solve an algebraic system of two inequalities where the unknown is given in terms of the difference of two mild solutions starting from the same initial condition but in two different norms. The solution of this system is zero and this is what allows to conclude the uniqueness of solutions. We believe that our result of existence of solutions  can be generalized to the Navier-Stokes equations although careful calculations have to be performed on the nonlinear term which is the main difference with the current result. We might have to work in slightly different spaces, and this will be done in the forthcoming paper  \cite{BeGaSch15}.

\smallskip
Articles dealing with pathwise solutions for quite general stochastic ordinary differential equations driven by a multiplicative fractional--Brownian motion are, e.g., \cite{NuaRas02} and \cite{GMS08}. In the infinite dimensional context, there are also articles studying the existence of pathwise solutions, like \cite{NuaVui06} (dealing with variational solutions) and \cite{MasNua03}, \cite{GLS09}, \cite{diop} and \cite{ChGGSch12}, for the mild solution. In  these papers the Hurst--parameter $H\in (1/2,1)$, the diffusion and the drift are assumed to be Lipschitz--continuous and the existence of solutions is proved using pathwise arguments through the fractional integrals.
There is an extensive literature for fluid flows  driven by a Brownian--motion but  only a few with  a fractional Brownian--motion.  In \cite{CaQTu} another fluid model is considered driven by
a fractional Brownian--motion with Hurst--parameter bigger than 1/2. In particular the authors find a local solution of the 3D Navier--Stokes--equation by using the Young--integral. In \cite{viens} the 2D Navier--Stokes--equation is studied driven by a fractional Brownian--motion with more general Hurst--parameters. However, the considered noise is additive.

\smallskip

An interesting advantage of considering the existence of pathwise solutions for the stochastic shell--model is that they will generate a random dynamical system, which gives us the possibility to an intensive asymptotic analysis of \eqref{SHELL}. In particular, this is the foundation to show the existence of random attractors and the analysis of their structure. In the forthcoming paper \cite{BeGaSch15} the dynamics of the stochastic shell--model is investigated by using the random dynamical system theory. We would like to point out that, despite the fact that there are similarities between the 2D Navier--Stokes--equation and the shell-model, more effort and more involved techniques will be necessary to obtain similar results for the stochastic 2D Navier--Stokes--equation than the ones considered in \cite{BeGaSch15}. Let us also mention that the generation of a random dynamical system as well as the study of the corresponding random attractor for another kind of stochastic evolution equations with multiplicative fractional noise have been very recently investigated in the papers \cite{ChGGSch12}, \cite{ChGGSch14}, and \cite{GMS08}.

\smallskip
\smallskip
The paper is organized as follows: in Section 2 we introduce the functional analytical framework. In Section 3, we define the fractional derivatives and the stochastic integral using some type of generalized Young--integrals. In Section 4, we introduce the different assumptions on the diffusion and give the definitions of the different solutions. In
Section 5, we prove that the system driven by an smoother path has a unique weak solution, that it is also a mild solution. Furthermore, we obtain some fundamental uniform estimates for the solution of the system driven by such a kind of smooth path. In Section 6, thanks to these uniform estimates and a compactness reasoning we construct a unique pathwise mild solution to the shell--model having a fractional Brownian--motion as driving path. Section 7 is devoted to an example of a particular diffusion fitting the assumptions required for developing the abstract framework. Finally, Section 8 contains the proofs of some results that have been used in different sections of the paper.

\smallskip

As usual we denote by $c$ a positive constant that can change their value from line to line.

\section{Preliminaires}\label{S-FS}

\subsection{Spaces and operators}\label{s1}
For any $\alpha \in \mathbb R$,  let us introduce the following spaces, see Constantin et al. \cite{clt} for the details,
\begin{equation*}
V_{\alpha}=\{u=(u_1, u_2, \ldots) \in \mathbb {C}^\infty:
 \sum_{n=1}^\infty k_n^{4\alpha}|u_n|^2<\infty \}\footnotemark.
\end{equation*}
This is a separable Hilbert--space with scalar product
$( u,v)_{V_\alpha}=\sum_{n=1}^\infty k_n^{4\alpha} u_n
\bar v_n$.\footnotetext{Here there is an important difference w.r.t. the notation of spaces in \cite{clt} and \cite{clt2}.}
 Denote by  $\|\cdot \|_{V_\alpha}$  its norm.
We have the  compact embedding
\begin{equation*}
 V_{\alpha_1}\subset V_{\alpha_2} \qquad \text{ if } \alpha_1>\alpha_2.
\end{equation*}
Let us denote by $V:=V_{0}$ and its norm simply by $\|\cdot\|$ and its scalar product by $( \cdot, \cdot )_V$.

\smallskip

Let $A: D(A)=V_1\to V$ be the linear unbounded operator defined as

\begin{equation*}
 A: (u_1, u_2, \ldots) \mapsto (-\nu k_1^2 u_1,-\nu k_2^2 u_2,\ldots).
\end{equation*}
For simplicity let us set $\nu=1$.
It is known that $A$ generates an analytic semigroup $S(\cdot)$ which follows from the Lax-Milgram lemma, see Sell and You \cite{SelYou02} Theorem 36.6, and this semigroup is exponentially stable.  Furthermore, $V_\alpha =D(A^\alpha)$ and $(u,v)_{V_\alpha}=(A^{\alpha}u, A^{\alpha}v)_V$, $u,v\in V_{\alpha}$.

Let $L(V_\delta,V_\gamma)$ denote the space of linear continuous operators from $V_\delta$ into $V_\gamma$.
As usual, $L(V)$ denotes $L(V,V)$.
The following properties are well known for analytic semigroups and their generators: for $\zeta\ge \alpha$ there exists a constant $c>0$ such that
\begin{eqnarray}
  |S(t)|_{L(V_\alpha, V_{\zeta})}=|A^\zeta S(t)|_{L(V_\alpha,V)}\le
  \frac{c}{t^{\zeta-\alpha}}e^{-\lambda t},\quad t>0\label{eq1},
  \end{eqnarray}
  \begin{eqnarray}
 |S(t)-{\rm id}|_{L(V_{\sigma+\nu},V_{\theta+\nu})} &\le c
t^{\sigma-\theta}, \quad {\rm for }\  \sigma\in
[\theta,1+\theta],\quad \nu\in \RR \label{eq2},
\end{eqnarray}
where $\lambda$ in (\ref{eq1}) is a positive constant, see for instance Pazy \cite{Pazy} Theorem 2.6.13. From these inequalities, for $\nu,\,\eta\in [0,1]$, $\xi,\delta\in \RR$ such that $\delta\leq \zeta+\nu$, there exists a $c>0$ such that for $0\leq q\leq r\leq s\leq t$,
\begin{align}\label{eq30}
\begin{split}
|S(t-r)-S(t-q)|_{L(V_{\delta},V_{\zeta})}\le c(r-q)^\nu(t-r)^{-\nu-\zeta+\delta},\\
 |S(t-r)- S(s-r)-S(t-q)+S(s-q)|_{L(V)}\leq
c(t-s)^{\nu}(r-q)^{\eta}(s-r)^{-(\nu+\eta)}.\end{split}
\end{align}

\smallskip

Define the bilinear operator $B:{\mathbb C}^\infty\times {\mathbb
C}^\infty\to {\mathbb C}^\infty$ as
\[
 B(u,v)=-(b_1(u,v), b_2(u,v),\ldots)
\]
where the components $b_i$ satisfy (\ref{incomp}).

$B$ is well defined  when its domain is
$V_{1/2} \times V$ or $V \times V_{1/2}$ (see  \cite{clt}), that is,
$ B:  V_{1/2} \times V\to V$  and
$ B:  V \times V_{1/2}\to V$ are bounded operators. The operator $B$ enjoys the following properties
\begin{align*}
\begin{split}
( B(u,v),w)_V&=-( B(u,w),v)_V,\quad  u\in V_{1/2}, \quad v, w\in V,\\
( B(u,v),w)_V&=-( B(u,w),v)_V,\quad  u\in V, \quad  v, w\in V_{1/2}.
\end{split}
\end{align*}
 As a consequence, we also have that
\begin{equation}\label{skew2}
( B(u,v),v)_V=0, \quad  u\in V,  \quad v\in V_{1/2}.
\end{equation}

Moreover, we extend the
result of Constantin {\it et al.} \cite{clt} to more general spaces:

\begin{lemma}\label{Bgenerale}
For any $\alpha_1, \alpha_2, \alpha_3 \in \mathbb R$
$$
 B:V_{\alpha_1}\times V_{\alpha_2} \to V_{-\alpha_3} \;
 \text{ with }  \alpha_1+\alpha_2+\alpha_3\ge\frac12
$$
and there exists a constant $c$ depending on the
$\alpha_j$'s such that
\[
 \|B(u,v)\|_{V_{-\alpha_3}}\le c \|u\|_{V_{\alpha_1}} \|v\|_{V_{\alpha_2}},
 \quad   u \in V_{\alpha_1}, \quad v \in V_{\alpha_2}.
\]
\end{lemma}
The proof of this result follows by Proposition 1 of Constantin {\it et al.} \cite{clt2}, and Bessaih {\it et al.} \cite{Bessaih-Ferrario1} and hence we omit it here.\\

Let $C([s,t];V_\mu)$ be the space of continuous functions on $[s,t]$ with values in $V_\mu$ and with the usual norm $\|\cdot\|_{C,\mu}$ (or $\|\cdot\|_{C,s,t,\mu}$ when we want to stress the interval). In the particular case that $\mu=0$, we simply write $\|\cdot\|_{C}$ (or $\|\cdot\|_{C,s,t}$ respectively).
For $\beta\in (0,1]$ we denote by $C^\beta([s,t];V_\mu)$ the space of H{\"o}lder--continuous functions on $[s,t]$ and with values in $V_\mu$, equipped with the norm
\begin{equation*}
  \|u\|_{\beta,\mu}=\|u\|_{C,\mu}+|||u|||_{\beta,\mu},\quad |||u|||_{\beta,\mu}:=\sup_{s\le p<q\le t}\frac{\|u(q)-u(p)\|_{V_\mu}}{(q-p)^\beta}.
\end{equation*}

In particular, for the case $\beta=1$ this is the space of Lipschitz--continuous functions.

The spaces $L^p(s,t;V_\mu),\,p\in [1,\infty]$ have the standard meaning with the usual norms.

As we have mentioned above, sometimes it is important to consider the above norms on different time intervals $[s,t]$, thus in those cases the time interval will be indicated in the index of the norm.

\smallskip

For the previous spaces the following compactness theorem holds true:

\begin{theorem}\label{t1}
(i) For $\alpha,\delta>0$, $L^2(s,t;V_{\alpha})\cap C^\beta([s,t];V_{-\delta})$ is compactly embedded into $L^2(s,t;V)\cap C([s,t];V_{-\delta})$.

\smallskip

(ii) For $0 \leq \delta_1< \delta_2$ and $0\le\beta_1<\beta_2\le 1$ the space
$C^{\beta_2}([s,t];V_{-\delta_1})$ is compactly embedded into $C^{\beta_1}([s,t];V_{-\delta_2}) $.
\end{theorem}

For the first part see Vishik and Fursikov \cite{FurVis} Chapter IV Theorem 4.1. For the second part we refer to Maslowski and Nualart \cite{MasNua03} Lemma 4.5. Indeed, we have the compact embedding $V_{-\delta_1}\subset V_{-\delta_2}$.\\

We  now rewrite the equation \eqref{SHELL} in an abstract form
\begin{equation}\label{abstract}
du(t)=\left(Au(t)+ B(u(t),u(t))\right)dt+ G(u(t))d\omega(t).
\end{equation}
where $G$ is a nontrivial diffusion term representing the external force, and which assumptions will be describe in Section \ref{ms} below. Here $\omega$ represents a {\em path} in $C^{\beta^\prime}([0,T];V)$, with $\beta^\prime>1/2$, or in particular, a fractional Brownian--motion with Hurst--parameter $H\in (1/2,1)$, see the definition in Section \ref{s2}. This stochastic evolution equation has therefore a multiplicative noise. In what follows we will describe the type of stochastic integral we are going to consider, which will allow us to give an appropriate meaning to \eqref{abstract}.

\section{Integrals in Hilbert--spaces for H{\"o}lder-continuous integrators with H{\"o}lder exponents greater than $1/2$}\label{s2}

In this section we are concerned with the definition of the following infinite dimensional integral
\begin{equation*}
    \int_{T_1}^{T_2} Zd\omega,
\end{equation*}
where $\omega$ is a H{\"o}lder-continuous function with H\"older exponent $\beta^\prime>1/2$ and $Z$ is an appropriate integrand. We follow  the recent definition given by Chen {\it et al.} \cite{ChGGSch12}, and for the sake of completeness, next we shall borrow the main steps of their construction.

We start by considering an abstract separable Hilbert--space $\tilde V$, then for $0<\alpha<1$ and general measurable functions $Z:[T_1,T_2]\to \tilde V$ and $\omega:[T_1,T_2]\to  V$, we define the following fractional derivatives
\begin{align*}
    D_{{T_1}+}^\alpha Z[r]&=\frac{1}{\Gamma(1-\alpha)}\bigg(\frac{Z(r)}{(r-T_1)^\alpha}+\alpha\int_{T_1}^r\frac{Z(r)-Z(q)}{(r-q)^{1+\alpha}}dq\bigg)\in \tilde V,\,\\
    D_{{T_2}-}^{1-\alpha} \omega_{T_2-}[r]&=\frac{(-1)^{1-\alpha}}{\Gamma(\alpha)}
    \bigg(\frac{\omega(r)-\omega(T_2-)}{(T_2-r)^{1-\alpha}}
    +(1-\alpha)\int_r^{T_2}\frac{\omega(r)-\omega(q)}{(q-r)^{2-\alpha}}dq\bigg)\in
     V,
\end{align*}
where $ \omega_{T_2-}(r)= \omega(r)- \omega(T_2-)$, being $\omega(T_2-)$ the left side limit of $\omega$ at $T_2$. Here $\Gamma(\cdot)$ denotes the Gamma function.

Let us start with the case when the integrand $z$ and the integrator $\zeta$ are one-dimensional.
Suppose that $z(T_1+),\,\zeta(T_1+),\,\zeta(T_2-)$ exist, being respectively the right side limit of $z$ at $T_1$ and the right and left side limits of $\zeta$ at $T_1,\,T_2$, and that $z \in I_{T_1+}^\alpha (L^p(T_1,T_2;\mathbb R)),\, \zeta_{T_2-} \in
I_{T_2-}^{\alpha} (L^{p^\prime}(T_1,T_2; \mathbb R))$ with $1/p+1/{p^\prime}\le 1$ and $\alpha p<1$ (the definition of these spaces can be found, for instance, in Samko {\it et al.} \cite{Samko}). Then following Z\"ahle \cite{Zah98} we define
\begin{align}\label{eq11}
    \int_{T_1}^{T_2} zd\zeta&=(-1)^\alpha\int_{T_1}^{T_2} D_{T_1+}^\alpha z[r]D_{T_2-}^{1-\alpha}\zeta_{T_2-}[r]dr.
\end{align}
Suppose now that $\zeta$ is Lipschitz--continuous. Then $\zeta$ generates a signed measure $d\zeta$ and $\zeta\in I_{T_2-}^{\alpha} (L^{p^\prime}(T_1,T_2; \mathbb R))$. Therefore, in this situation the integral
\begin{equation*}
  \int_{T_1}^{T_2}zd\zeta
\end{equation*}
can be expressed by \eqref{eq11}.

\medskip

Let $\hat V$ be a separable Hilbert--space endowed with the norm $\|\cdot\|_{\hat V}$ and consider the separable Hilbert--space $L_2(V,\hat V)$ of Hilbert-Schmidt--operators from $V$ into $\hat V$ with the norm $\|\cdot\|_{L_2(V,\hat V)}$ and inner product $(\cdot,\cdot)_{L_2(V,\hat V)}$. Let $(e_i)_{i\in\NN}$ and $(f_i)_{i\in\NN}$ be a complete orthonormal basis of $V$ and $\hat V$, resp. A base in $L_2(V,\hat V)$ is given by
\begin{align*}  E_{ij}e_k=\left\{\begin{array}{lcl}
    0&:& j\not= k\\
    f_i &:& j= k.
    \end{array}
    \right.
\end{align*}
Let us consider now mappings $Z:[T_1,T_2]\to L_2(V,\hat V)$ and $\omega:[T_1,T_2]\to  V$. Suppose that $z_{ji}=(Z,E_{ji})_{L_2(V,\hat V)}\in I_{T_1+}^\alpha (L^p(T_1,T_2;\mathbb R))$ and $z_{ji}(T_1+)$ exists and $\alpha p<1$. Moreover,  let us also assume that $\zeta_{iT_2-}=(\omega_{T_2-}(t),e_i)_V\in I_{T_2-}^{1-\alpha} (L^{p^\prime}(T_1,T_2; \mathbb R))$ such that $1/p+1/p^\prime\le 1$, and the mapping
\begin{equation*}
   [T_1,T_2]\ni r\mapsto  \|D_{T_1+}^\alpha Z[r]\|_{L_2(V,\hat V)}\|D_{T_2-}^{1-\alpha} \omega_{T_2-}[r]\|\in L^{1}(T_1,T_2;\RR).
\end{equation*}

We  introduce
\begin{align}\label{eq3}
\begin{split}
    \int_{T_1}^{T_2} Z d\omega&:= (-1)^\alpha\int_{T_1}^{T_2} D_{T_1+}^\alpha Z[r]D_{T_2-}^{1-\alpha}\omega_{T_2-}[r]dr\\
&:=(-1)^\alpha\sum_{j=1}^\infty\bigg(\sum_{i=1}^\infty \int_{T_1}^{T_2}
    D_{T_1+}^{\alpha}z_{ji}[r]D_{T_2-}^{1-\alpha}\zeta_{iT_2-}[r]dr \bigg) f_j.
\end{split}    
\end{align}
This last equality is well defined due to the fact that Pettis' theorem and the separability of $V$ ensure that the integrand is weakly measurable and hence measurable. Moreover, the norm of the above integral is given by
\begin{align*}
    \bigg\|\int_{T_1}^{T_2}Zd\omega\bigg\|_{\hat V}
    &=\bigg(\sum_{j=1}^\infty \bigg|\sum_{i=1}^\infty \int_{T_1}^{T_2}D_{T_1+}^{\alpha}z_{ji}[r]D_{T_2-}^{1-\alpha}\zeta_{iT_2-}[r]dr\bigg|^2\bigg )^\frac12\\
    &\le \int_{T_1}^{T_2}\|D_{T_1+}^\alpha Z[r]\|_{L_2(V,\hat V)}\|D_{T_2-}^{1-\alpha} \omega_{T_2-}[r]\| dr.
\end{align*}

The next result, which proof can be found in \cite{ChGGSch12}, considers the definition of the above integral when having suitable H\"older continuous integrator and integrand functions:

\begin{lemma}\label{l3} Suppose that $Z\in C^{\beta}([T_1,T_2];L_2(V,\hat V))$ and $\omega\in C^{\beta^\prime}([T_1,T_2];V)$ with $1-\beta^\prime<\alpha<{\beta}$.  Then
\[
\int_{T_1}^{T_2} Z d\omega\in  \hat V
\]is well-defined in the sense of (\ref{eq3}). Also, there exists a constant $c$ depending only on $T_2,\,\beta,\,\beta^\prime$ such that
\begin{align*}
\bigg\|\int_{T_1}^{T_2} Z d\omega\bigg\|_{\hat V}&\le
     c \|Z\|_{C^{\beta}([T_1,T_2];L_2(V,\hat V))} |||\omega|||_{\beta^\prime,T_1,T_2}(T_2-T_1)^{{\beta^\prime}}.
\end{align*}
\end{lemma}
Moreover, the above integral with driving path $\omega$ is well-defined even though the integrand is locally H\"older-continuous, which will be the case in the next sections when the semigroup $S$ is part of the integrand, see \cite{ChGGSch12} for the proof of this assertion.

\smallskip
In the following we would like to consider the above integrals when the integrator is a noise given by a fractional Brownian--motion (fBm). An one--dimensional fBm is a centered Gau\ss--process given by the auto-covariance
\begin{equation*}
  R(s,t)=\frac12(t^{2H}+s^{2H}-|t-s|^{2H})
\end{equation*}
where $H\in (0,1)$ is the so-called Hurst--parameter. The value $H=1/2$  determines
a Brownian--motion, which is a martingale and a Markov--process with independent increments. When $H\not=1/2$ these properties do not hold.

An fBm can be also defined in a separable Hilbert--space. By the following construction we obtain such an infinite-dimensional noise with values in $V$: let $(\zeta_i)_{i\in\NN}$ be a iid-sequence of fBm in $\RR$
having the same Hurst--parameter $H$. Then
\begin{equation*}
  t\to \omega(t):=\sum_{i=1}^\infty q_i^\frac12\zeta_i e_i,
\end{equation*}
where $(q_i)_{i\in\NN}\in l_2$, defines an fBm with values in $V$ and with auto-covariance
\begin{equation*}
  \frac12Q(t^{2H}+s^{2H}-|t-s|^{2H}),
\end{equation*}
where the operator $Q$ of diagonal form is defined by
\begin{equation*}
  (e_i,Qe_j)_{V}=\delta_{ij}q_i.
\end{equation*}
One very important property that will be crucial in this paper is that, thanks to Kolmogorov's theorem, the stochastic process $\omega$ has a $\gamma$--H{\"o}lder--continuous version for any $\gamma<H$, see Theorem 1.4.1 in Kunita \cite{Kunita90}.
For simplicity we restrict ourself to a real fBm. However, taking two one--dimensional independent real fBm $\zeta^1,\,\zeta^2$ then we could construct
a one--dimensional complex fBm: $\zeta:=1/\sqrt{2}(\zeta^1+i\zeta^2)$. Then by the above formula we could construct a complex fBm $\omega$ in $V$.
\\

\begin{remark} \label{sep} For our further purposes we need the fBm $\omega$ to be piecewise linear approximated. As one can check later, we will use the property that given $\omega$ we can find a sequence of piecewise linear continuous functions $\omega_n$ converging to $\omega$ in a H\"older--continuous space. However, the space of H\"older--continuous functions is not separable, but we can modify it in such a way that the modified space is: for $\beta^\prime<\gamma<H$
\begin{align*}
C^{0,\beta^\prime}([0,T];V):=\bigg\{\omega\in C^{\beta^\prime}([0,T];V): \lim_{\delta \to 0} \sup_{|s_1-s_2|<\delta, [0,T] \ni s_1\not =s_2} \frac{\|\omega(s_1)-\omega(s_2)\|}{(s_1-s_2)^{\beta^\prime}}=0\bigg\},
\end{align*}
is a separable space since $V$ is itself separable (see \cite{FH13}, \cite{FV10} and \cite{ducsigsch}). It is easy to see that $C^{\gamma}([0,T];V) \subset C^{0,\beta^\prime}([0,T];V)$, and therefore this latter space is the one that we should take when considering the fBm. Hence, in what follows we shall assume that the path $\omega$ in $C^{\beta^\prime}([0,T];V)$ can be piecewise linear approximated by a sequence $(\omega_n)_{n\in \mathbb N}$ converging in $C^{\beta^\prime}([0,T];V)$, because we assume that $\omega\in C^{\gamma}([0,T];V)$  with $\gamma <H$, but this statement must be understood according to the sense given above.
\end{remark}

\section{Definition of a solution of the stochastic shell--model}\label{ms}

In this section we would like to formulate conditions ensuring that \eqref{abstract} has a global unique solution.

We emphasize that we have to formulate a definition of solution which is appropriate in our context: on the one hand, the driving function belongs to $C^{\beta^\prime}([0,T];V)$ for a $\beta^\prime\in (1/2,1)$, which means that we cannot define integrals by using the standard integration theory of bounded-variation integrators. In particular our situation here covers the case when the driving function is given by an fBm in $V$ with Hurst--parameter $H>1/2$. On the other hand, the bilinear operator $B$ will be the responsible of having to deal with {\em non--Lipschitz--coefficients} in this model.

\smallskip

We now formulate the assumptions for the diffusion operator $G$ of \eqref{abstract}. In what follows,
we choose a constant $\delta>0$ which will be determined later.

{\bf Assumption (G)} Assume that the mapping $G: V_{-\delta}\to L_2(V)$ is bounded and twice continuously Fr\'echet--differentiable with bounded first and second derivatives $DG(u)$ and $D^2G(u)$, for $u\in V_{-\delta}$. Let us denote, respectively, by $c_G$, $c_{DG}$ and $c_{D^2G}$
the bounds for $G$, $DG$ and $D^2G$. Then, for $u\in V_{-\delta}$
\begin{equation*}
\|G(u)\|_{L_2(V)} \le c_G.
\end{equation*}
Furthermore, for $u_1,u_2\in V_{-\delta}$,
\begin{equation*}
\|G(u_1)-G(u_2)\|_{L_2(V)} \leq c_{DG}\|u_1-u_2\|_{V_{-\delta}},
\end{equation*}
and for $u_1, u_2, v_1, v_2 \in V_{-\delta}$,
\begin{align}\label{eq12}
\begin{split}
    \|G(u_1)-G(v_1)&-(G(u_2)-G(v_2))\|_{L_2(V)}\\
    \le & c_{DG}\|u_1-v_1-(u_2-v_2)\|_{V_{-\delta}}+c_{D^2G} \|u_1-u_2\|_{V_{-\delta}}(\|u_1-v_1\|_{V_{-\delta}}+\|u_2-v_2\|_{V_{-\delta}}).
\end{split}
\end{align}
Notice that $DG: V_{\delta} \mapsto L_2(V\times V_{-\delta},V)$ is a bilinear mapping whereas $D^2G$ a trilinear mapping.

\bigskip

In this paper we shall look at the existence and uniqueness of a solution of \eqref{abstract} according to the next definition:
\begin{definition}\label{def1}
Let $1/2<\beta<\beta^\prime\le 1$ and let $\omega\in C^{\beta^\prime}([0,T],V)$, $u_0\in V$ and $\delta\in (\beta,1)$. A function
$u$ is said to be a {\it mild} solution to \eqref{abstract} over the interval $[0,T]$  associated to the initial condition $u_{0}$ if
\begin{equation*}
u\in C([0,T], V)\cap L^{2}(0,T,V_{1/2})\cap C^{\beta}([0,T],V_{-\delta})
\end{equation*}
and such that
\begin{equation}\label{eq7}
  u(t)=S(t)u_0+\int_0^tS(t-r)B(u(r),u(r))dr+\int_0^tS(t-r)G(u(r))d\omega(r)
\end{equation}
for every $t\in [0,T]$.
\end{definition}

\begin{remark}\label{r1}
Note that the first integral in \eqref{eq7} is well defined in $V$ because of the fact that $u\in C([0,T],V)\cap L^2(0,T,V_{1/2})$ and Lemma \ref{Bgenerale}. The stochastic integral in \eqref{eq7} must be understood in $V$ according to the definition given in Section \ref{s2}.
\end{remark}

We stress that we are interested in finding a mild solution for (\ref{abstract}). Following \cite{NuaVui06} we could also consider weak solutions for our problem. Nevertheless for $u\in L^\infty(0,T,V)\cap L^2(0,T,V_{1/2})$ we have that $B(u,u)$ is sufficiently regular so that we can work with mild solutions.

When $\omega$ is regular, we can also interpret the solution in the following weak sense:
\begin{definition}\label{def1b}
Assume that $\omega$ is piecewise linear continuous in $[0,T]$ with values in $V$ and $u_0\in V$. We say that
$u$ is a {\it weak} solution to \eqref{abstract} over the interval $[0,T]$ associated to the initial condition $u_{0}$ if
\begin{equation*}
u\in C([0,T], V)\cap L^{2}(0,T,V_{1/2})
\end{equation*}
and such that
\begin{equation}\label{eq5}
( u(t),\phi)_V+\int_{0}^{t}( A^{1/2}u(s),A^{1/2}\phi)_V ds
-\int_{0}^{t}( B(u(s),u(s)), \phi)_V ds=( u_0,\phi)_V
+\int_{0}^{t} (G(u(s))\omega^\prime (s),\phi)_V ds
\end{equation}
\end{definition}
holds for every $\varphi\in V_{1/2}$ and $t\in [0,T]$.

\section{Solutions of the stochastic shell--model for piecewise linear continuous path noise}\label{s4}
In this section we assume that $\omega$ is a piecewise linear continuous function. This case is the foundation for studying the more general case which will be treated in the next section. Indeed, in further sections given $\omega\in C^{\beta^\prime}([0,T],V)$ we shall consider a sequence $(\omega_n)_{n\in\NN}$ of piecewise linear continuous paths converging to $\omega\in C^{\beta^\prime}([0,T],V)$, see Remark \ref{sep}. As we cannot assume that the sequence $(\omega_n^\prime )_{n\in\NN}$ in uniformly bounded in $L^\infty([0,T],V)$, we will have to construct  uniform a priori estimates for the solutions to equations driven by $\omega_n$, which will be based on uniform estimates of $(\omega_n)_{n\in\NN}$ with respect to the $C^{\beta^\prime}$-norm. \\

We start by studying the existence of solutions for the stochastic Shell--model having this kind of regular driving function:

\smallskip

\begin{prop}\label{w-smooth}
Assume that $\hat\beta\in (1/2,1)$, $\delta \in (\hat \beta, 1)$, $u_{0}\in V$, $\omega$ is a piecewise linear continuous function and $G$ and satisfies the assumption {\bf (G)}. Then, there is a global unique weak solution $u$ for equation \eqref{abstract} in the sense of Definition \ref{def1b}.
\end{prop}
\begin{proof}

The proof is very classical but, for the the sake of completeness, we will sketch it here.
 Let us denote by $P_{n}$ the projection operator in $V$ onto the space spanned by
$e_{1}, e_{2}, \dots, e_{n}$. Then, the Galerkin--approximations $(u_n)_{n\in \NN}$ to problem (\ref{abstract}) are solutions of the finite-dimensional
systems
\begin{equation}\label{abstract-n}
du_{n}(t)=( Au_{n}(t)+ P_{n}B(u_{n}(t),u_{n}(t)))dt+ P_{n}G(u_{n}(t))\omega'(t)dt.
\end{equation}
On the other hand, if $G^*$ denotes the adjoint operator of $G$, taking the scalar product of \eqref{abstract-n} by $u_{n}$, using the property (\ref{skew2}) and assumption {\bf (G)}, we get that

\begin{equation*}
\begin{split}
\frac{1}{2}\frac{d}{dt}\|u_{n}(t)\|^{2}+\|u_{n}(t)\|_{V_{1/2}}^{2}&\leq
\left|(  P_{n}G(u_{n}(t))\omega'(t), u_{n}(t))_V\right|\leq \left|(  \omega'(t), G^{*}(u_{n}(t)) u_{n}(t))_V\right| \\
&\leq \|\omega'(t)\| \|G^{*}(u_{n}(t)) u_{n}(t)\|\leq c_G\|\omega'(t)\| \|u_{n}(t)\|\\
&\leq \frac{c_G^2}{2}\|\omega'(t)\|^{2}+\frac{1}{2}\|u_{n}(t)\|^{2}.
\end{split}
\end{equation*}
Hence,  using the Gronwall lemma yields that
\begin{equation*}
\sup_{t\in [0,T]}\|u_{n}(t)\|^{2}\leq c(\|u_n(0)\|, \|\omega'\|_{L^\infty(0,T,V)}^{2}, T)
\end{equation*}
for an appropriate positive constant $c$, and consequently we also have
\begin{equation*}
\int_{0}^{T}\|u_{n}(t)\|_{V_{1/2}}^{2}dt\leq c,
\end{equation*}
uniformly in $n$.

Also, by classical arguments, we get that  $(u_{n})_{n\in\NN}$ is bounded in
$C^{\hat \beta}([0,T], V_{-\delta})$. In fact,  since $u_n\in L^\infty(0,T,V)$ and, in particular, $\delta >1/2$, it follows by Lemma \ref{Bgenerale}
\begin{equation*}
  \sup_{0\le s<t\le T}\frac{\int_s^t\|A^{-\delta}B(u_n(r),u_n(r))\|dr}{(t-s)^{\hat \beta}}\le c\sup_{0\le s<t\le T}\frac{\int_s^t\|A^{-\frac12}B(u_n(r),u_n(r))\|dr}{(t-s)^{\hat \beta}}\le c T^{1-\hat \beta}\|u_n\|_{L^\infty(0,T,V)}^2<\infty,
\end{equation*}
and by {\bf (G)} we arrived at
\begin{equation*}
  \sup_{0\le s<t\le T}\frac{\int_s^t\|A^{-\delta}G(u_n(r))\omega^\prime(r)\|dr}{(t-s)^{\hat \beta}}\le c \sup_{0\le s<t\le T}\frac{\int_s^t\|G(u_n(r))\|_{L_2(V)}\|\omega^\prime(r)\|dr}{(t-s)^{\hat \beta}}\le c c_G T^{1-\hat \beta}\|\omega^\prime\|_{L^\infty(0,T,V)}<\infty.
\end{equation*}
Moreover, applying the interpolation inequality (see \cite{SelYou02}, Theorem 37.6), we know that there exists a constant $c=c(\delta)\geq 1$ such that
$$\|A^{1-\delta} v\| \leq c \|A^0 v\|^{2\delta-1} \|A^{1/2} v\|^{2-2\delta}\quad \text{ for all } v\in V,$$
and therefore
\begin{align*}
 & \sup_{0\le s<t\le T}\frac{\int_s^t\|A^{-\delta}Au_n(r)\|dr}{(t-s)^{\hat \beta}}\le c \sup_{0\le s<t\le T} \frac{\int_s^t \|u_n(r)\|^{2\delta-1} \|A^{1/2} u_n(r)\|^{2-2\delta}dr}{(t-s)^{\hat \beta}}\\
 \le & c \|u_n\|_{L^\infty(0,T,V)}^{2\delta-1} \sup_{0\le s<t\le T} \frac{(\int_s^t dr)^{\delta} (\int_s^t \|A^{1/2} u_n(r)\|^{2}dr)^{1-\delta}}{(t-s)^{\hat \beta}}
 \le  c T^{\delta-\hat \beta}  \|u_n\|_{L^\infty(0,T,V)}^{2\delta-1} \|u_n\|_{L^2(0,T,V_{1/2})}^{2-2\delta} <\infty.
\end{align*}

Hence, by the compactness Theorem \ref{t1} (i) we get a subsequence, still denoted by $(u_{n})_{n\in\NN}$, that converges  strongly in  $L^2(0,T, V)\cap C([0,T],V_{-\delta})$ to some limit $u$.  Since $(u_n)_{n\in\NN}$ is bounded in $L^\infty(0,T, V)\cap L^2(0,T, V_{{1/2}})$
this sequence is relatively weak-star compact in $L^\infty(0,T, V)$ and relatively weak compact in $L^2(0,T,V_{{1/2}})$.
As a consequence, the limit
$u\in  L^\infty(0,T, V)\cap L^2(0,T, V_{{1/2}})$. Now, it remains to prove that the limit $u$ is a solution to the system \eqref{abstract} according to the Definition \ref{def1b}. Indeed, assuming that $u_{n}$ is solution in the sense of Definition \ref{def1b}, we can pass to the limit on each term. Furthermore, the regularity of $u$ implies that the right hand side of (\ref{eq5}) as well as the last two terms of the left hand side of (\ref{eq5}) are in $C([0,T], V)$, hence $u\in C([0,T], V)$. For similar limit considerations we refer to Constantin {\it et al.} \cite{clt}.
\end{proof}

Moreover, we have the following result about mild solutions:
\begin{prop}\label{prop}
Under the same hypotheses than in Proposition \ref{w-smooth}, every weak solution $u$ to \eqref{abstract} is a mild solution, that is, $u\in C([0,T], V)\cap L^{2}(0,T,V_{{1/2}})\cap C^{\hat \beta}([0,T],V_{-\delta})$ and satisfies for every $t\in[0,T]$ the following integral formulation in $V$:
\begin{equation}\label{integral1}
u(t)=S(t)u_{0}+\int_{0}^{t}S(t-r)B(u(r),u(r))dr+\int_{0}^{t}S(t-r)G(u(r))\omega'(r)dr.
\end{equation}
\end{prop}

\begin{proof}
Suppose that $u$ fulfills \eqref{eq5}. Then
\begin{equation*}
  t\mapsto B(u(t),u(t))+G(u(t))\omega^\prime(t)\in L^2(0,T,V)
\end{equation*}
such that
\begin{equation*}
t\mapsto \int_{0}^{t}S(t-r)B(u(r),u(r))dr+\int_{0}^{t}S(t-r)G(u(r))\omega'(r)dr\in C([0,T],V),
\end{equation*}
see Pazy \cite{Pazy}, proof of Theorem 4.3.1. In addition, every Galerkin--approximation solution of \eqref{abstract-n} satisfies
\begin{align}\label{eq6}
\begin{split}
  ( u_n(t),\phi)_V&=( S(t)P_n u_n(0),\phi)_V+\int_{0}^{t}(S(t-r)P_nB(u_n(r),u_n(r)),\phi)_Vdr\\
  &+\int_{0}^{t}(S(t-r)P_nG(u_n(r))\omega'(r),\phi)_Vdr
  \end{split}
\end{align}
for every $\phi\in V$ and every $t\in [0,T]$. From the convergence of $(u_n)_{n\in \NN}$ in $L^2(0,T,V)$ and the boundedness in $L^2(0,T,V_{{1/2}})$ it follows
that the left hand side of \eqref{eq6} converges to
\begin{equation*}
  ( S(t)u_{0},\phi)_V+\int_{0}^{t}(S(t-r)B(u(r),u(r)),\phi)_Vdr+\int_{0}^{t}(S(t-r)G(u(r))\omega'(r),\phi)_Vdr
\end{equation*}
for every $t\in [0,T]$. On the other hand, from the proof of Proposition \ref{w-smooth} we know that $(u_n)_{n\in\NN}$ converges to $u$ in $C([0,T],V_{-\delta})$  and hence $u_n(t)$ converges to $u(t)$ in $V_{-\delta}$ for every $t\in[0,T]$.
Since the right hand side is in $V$ for every $t\in[0,T]$, $u(t)$ too.
Also, following the same reasoning than in Proposition \ref{w-smooth}, one can prove that $(u_n)_{n\in \mathbb N}$ is bounded in $C^{\gamma} ([0,T], V_{-\hat \delta})$ for $\gamma=\hat \beta+\varepsilon$ and $\hat \delta=\delta-\varepsilon$ for small enough $\varepsilon >0$ such that $\hat \delta >\gamma$. Then it suffices to apply Theorem \ref{t1} (ii) to conclude the proof.
\end{proof}

\smallskip

From now on, we often use the following property, which is a consequence of the definition of Beta function: for every $0\leq s<t\leq T$, $a,\,b>-1$,
\begin{equation}\label{prop}
  \int_s^t(r-s)^a(t-r)^b dr=c(t-s)^{a+b+1}
\end{equation}
where $c$ only depends on $a$ and $b$. \\

Next we develop a priori estimates that later we need to derive the existence of a solution for a general $\omega\in C^{\beta^\prime}([0,T],V)$.
We cannot use the estimate from Proposition \ref{w-smooth} because we do not have that the sequence $(\omega_n)_{n\in\NN}$ approximating
$\omega$ in $C^{\beta^\prime}([0,T],V)$ is in general uniformly bounded in $L^\infty(0,T,V)$. That is why in the following estimates $|||\omega|||_{\beta^\prime}$ appears.

\begin{lemma}\label{l5}
Assume that $1/2<\hat\beta<\beta^\prime$, $1-\beta^\prime<\alpha<\hat\beta$, $\delta \in (\hat \beta, 1)$, $u_{0}\in V$, $\omega$ is a piecewise linear continuous function and $G$ satisfies {\bf (G)}. Then, if $u$ is a weak solution to (\ref{abstract}) in the sense of Definition \ref{def1b}, there is a constant $c>0$ such that for $t\in [0,T]$
\begin{align}\label{E1}
\begin{split}
\|u(t)\|^{2}+
2 \int_{0}^{t}\|u(r)\|_{V_{1/2}}^{2}dr
& \leq \|u_{0}\|^{2}+c |||\omega|||_{\beta^\prime} t^{\beta^\prime}\|u\|_{C,0,t}+c|||\omega|||_{\beta^\prime} t^{\hat{\beta}+\beta^\prime}(1+\|u\|_{C,0,t})|||u|||_{\hat \beta,-\delta,0,t}.
\end{split}
\end{align}
\end{lemma}
\begin{proof}
Applying the formula of the square norm, see Teman \cite{Teman} Lemma III.1.2 , using the skew-symmetric property (\ref{skew2}), and finally integrating over $(0,t)$, this gives us  for every $t\in [0,T]$ the following energy inequality
\begin{equation*}
\|u(t)\|^{2}+2\int_{0}^{t}\|u(r)\|_{V_{1/2}}^{2}dr\leq \|u_{0}\|^{2}+
2\left|\int_{0}^{t}(G^{*}(u(r))u(r),\omega^\prime(r))_Vdr\right|.
\end{equation*}

The integral on the right hand side of the previous expression can be interpreted  in the sense of Section \ref{s2} using fractional derivatives.
Since for any $r$ we have $\|D_{0+}^\alpha G^\ast(u(r))u(r)\|<\infty$ the expression $D_{0+}^\alpha G^\ast(u(r))u(r)$ can be interpreted as an element in the space of Hilbert--Schmidt--operators $L_2(V,\RR)\simeq V$. Moreover, from the definition of the fractional derivative it is easy to derive that
\begin{align}\label{er}
\|D^{1-\alpha}_{t-}\omega_{t-}[r]\|\leq c |||\omega|||_{\beta^\prime} (t-r)^{\alpha+\beta^\prime-1},
\end{align}
and therefore we get
\begin{align*}
\begin{split}
&\left|\int_{0}^{t}(G^{*}(u(r))u(r),\omega^\prime(r))_Vdr\right|
\leq  c |||\omega|||_{\beta'} \int_{0}^{t} (t-r)^{\alpha+\beta'-1} \left(\frac{ \|G^{*}(u(r))u(r)\|}{r^{\alpha}} \right. \\
&\qquad \qquad \qquad +\left. \int_{0}^{r}\frac{\|G^{*}(u(r))u(r)-G^{*}(u(q))u(q)\|}{(r-q)^{1+\alpha}}dq\right)
dr.
\end{split}
\end{align*}

Trivially the boundedness of $G$ implies that $\|G^{*}(u)u\| \leq c_G\|u\|$ for $u\in V$ and therefore
\begin{equation*}
\frac{\|G^{*}(u(r))u(r)\|}{r^{\alpha}}\leq
\frac{c_G\|u\|_{C,0,t}}{r^{\alpha}},\quad  r\in [0,t].
\end{equation*}

The boundedness and the Lipschitz--continuity of $G$ imply

\begin{align*}
\int_{0}^{r}&\frac{\|G^{*}(u(r))u(r)-G^{*}(u(q))u(q)\|}{(r-q)^{1+\alpha}}dq\\
&\leq c_G\int_{0}^{r}\frac{\|u(r)-u(q)\|_{V_{-\delta}}}{(r-q)^{1+\alpha}}dq+\|u\|_{C,0,t}\int_{0}^{r}\frac{\|G^{*}(u(r))-G^{*}(u(q))\|_{L_2(V_{-\delta}, V)}}{(r-q)^{1+\alpha}}dq\\
&\leq (c_G+c_{DG}\|u\|_{C,0,t}) |||u|||_{\hat\beta,-\delta,0,t}\int_{0}^{r}(r-q)^{-1-\alpha+\hat\beta}dq\\
&= c(c_G+c_{DG}\|u\|_{C,0,t})  |||u|||_{\hat\beta,-\delta,0,t}r^{\hat\beta-\alpha}.
\end{align*}
Hence, for an appropriate $c>0$

\begin{align*}
\left|\int_{0}^{t}( G^{*}(u(r))u(r),\omega^\prime(r))_Vdr\right| &\leq c |||\omega|||_{\beta^\prime} t^{\beta^\prime}\|u\|_{C,0,t}+c |||\omega|||_{\beta^\prime} t^{\hat\beta+\beta^\prime}(1+\|u\|_{C,0,t})|||u|||_{\hat\beta,-\delta,0,t}.
\end{align*}
\end{proof}

\begin{lemma}\label{l7}
Under the same conditions of Lemma \ref{l5}, if $u$ is a solution to (\ref{abstract}) there exist constants $c,\,\bar c>0$ such that for $t\in [0,T]$
\begin{align}\label{beta}
\begin{split}
|||u|||_{\hat\beta,-\delta,0,t}&\leq \bar ct^{\delta-\hat\beta}\|u_0\|+ct^{1-\hat\beta}  \|u\|_{C,0,t}^2 +c |||\omega|||_{\beta'}t^{\beta^\prime-\hat\beta} ( 1+t^{\hat\beta}|||u|||_{\hat\beta,-\delta,0,t} ).
\end{split}
\end{align}
\end{lemma}

\begin{proof}
Consider \eqref{integral1} written as
\begin{equation*}
u(t)=S(t)u_{0}+A^{1/2}\int_{0}^{t}S(t-r)A^{-1/2}B(u(r),u(r))dr+\int_{0}^{t}S(t-r)G(u(r))\omega^\prime(r)dr.
\end{equation*}

Then the following splitting is considered:
\begin{align}\label{eq4}
\begin{split}
&A^{-\delta}(u(q)-u(p))=A^{-\delta}(S(q)-S(p))u_{0}+A^{-\delta+1/2}\int_{p}^{q}S(q-r)A^{-1/2}B(u(r),u(r))dr\\
&+
A^{-\delta+1/2}\int_{0}^{p}(S(q-r)-S(p-r))A^{-1/2}B(u(r),u(r))dr\\
&+A^{-\delta}\int_{p}^{q}S(q-r)G(u(r))\omega^\prime(r)dr+
A^{-\delta}\int_{0}^{p}(S(q-r)-S(p-r))G(u(r))\omega^\prime(r)dr\\
&=:I_{1}+I_{2}+I_{3}+I_{4}+I_{5}.
\end{split}
\end{align}

For the term related to the initial condition, due to the fact that $\delta \in (\hat\beta,1)$ and \eqref{eq1}, \eqref{eq2} we have
\begin{align*}
\sup_{0\leq p<q\leq t}\frac{\|I_{1}(p,q)\|}{(q-p)^{\hat\beta}}&\leq\sup_{0\leq p<q\leq t}\frac{\|A^{-\delta}(S(q-p)-{\rm Id})S(p)u_0\|}{(q-p)^{\hat\beta}} \leq \bar c \sup_{0\leq p<q\leq t}\frac{(q-p)^\delta\|u_0\|}{(q-p)^{\hat\beta}} \leq \bar c t^{\delta-\hat\beta}\|u_0\|.
\end{align*}

Moreover, due to Lemma \ref{Bgenerale} and taking into account that $V\subset V_{-\delta+1/2}$,
\begin{align}\label{eq18}
\begin{split}
\sup_{0\leq p<q\leq t}\frac{\|I_{2}(p,q)\|}{(q-p)^{\hat\beta}}&\leq\sup_{0\leq p<q\leq t}\frac{1}{(q-p)^{\hat\beta}}
\int_{p}^{q}\|A^{-\delta+1/2}S(q-r)A^{-1/2}B(u(r),u(r))\|dr\\
&\leq \sup_{0\leq p<q\leq t} \frac{c}{(q-p)^{\hat\beta}}
\int_{p}^{q} \|A^{-1/2}B(u(r),u(r))\|dr\\
&\leq  \sup_{0\leq p<q\leq t} \frac{c}{(q-p)^{\hat\beta}}(q-p)\|u\|_{C,0,t}^{2}
\leq c t^{1-\hat\beta} \|u\|_{C,0,t}^{2}.
\end{split}
\end{align}
For $I_{3}$, thanks to Lemma \ref{Bgenerale} and \eqref{eq30},

\begin{align}\label{eq19}
\begin{split}
\sup_{0\leq p<q\leq t} \frac{\|I_{3}(p,q)\|}{(q-p)^{\hat\beta}}&\leq \sup_{0\leq p<q\leq t}  \frac{c}{(q-p)^{\hat\beta}}
\int_{0}^{p}\|A^{-\delta+1/2}(S(q-p)-{\rm Id})S(p-r)A^{-1/2}B(u(r),u(r))\|dr\\
&\leq c\|u\|_{C,0,t}^{2} \sup_{0\leq p \leq t} \int_{0}^{p} (p-r)^{\delta-1/2-\hat\beta}dr\\
&\leq c t^{\delta+1/2-\hat\beta} \|u\|_{C,0,t}^{2}\le c^\prime t^{1-\hat\beta}\|u\|_{C,0,t}^2.
\end{split}
\end{align}

Similar estimates to those of $I_4,\,I_5$ can be found in \cite{ChGGSch12}. However, and for the completeness of the presentation, we also show these technical estimates in this paper, but we have shifted these calculations into the Appendix Section, see Lemma \ref{l4} above. In particular, in Lemma \ref{l4} we get
\begin{equation}\label{eq10}
\sup_{0\leq p<q\leq t} \frac{\|I_4(p,q)\|+\|I_{5}(p,q)\|}{(q-p)^{\hat\beta}} \leq ct^{\beta'-\hat\beta}|||\omega|||_{\beta'}
\left(1+t^{\hat\beta}|||u|||_{\hat\beta,-\delta,0,t}\right).
\end{equation}

Hence, collecting all the estimates for the expressions $I_j$ the inequality (\ref{beta}) is obtained.
\end{proof}
\begin{lemma}\label{corou}
Under the assumptions of Lemma \ref{l5}, if $u_n$ is a solution of \eqref{eq5} on $[0,T]$ with initial condition $u_0\in V$ and driven by a piecewise linear continuous path $\omega_n$ where $(\omega_n)_{n\in\NN}$ is bounded in $C^{\beta^\prime}([0,T],V)$, then $(u_n)_{n\in\NN}$ is uniformly bounded in $C^{\hat\beta}([0,T],V_{-\delta}) \cap C([0,T],V)$.
\end{lemma}

The proof of the previous result rests upon the technical Lemmas \ref{l6}--\ref{l8} whose proofs are presented into the Appendix section.

\begin{remark}
We emphasize that we consider H{\"o}lder--continuity with respect to the space $V_{-\delta}$ in the definition of a mild solution and in the results of this section, as well. The estimates in these results also make sense for smaller $\delta$. However, the initial condition $u_0$ is the responsible of having to consider $\delta\in (\hat \beta, 1)$, since in (\ref{beta}) the exponent in the term $t^{\delta-\hat\beta}$ multiplying $\|u_0\|$ must be positive.
\end{remark}

\section{Construction of solutions}
We are now able to construct solutions for the stochastic equation \eqref{abstract} and give the main result of this paper.
We consider a sequence of solutions $(u_n)_{n\in\NN}$ to \eqref{abstract} driven by $(\omega_n)_{n\in\NN}$, a sequence of piecewise linear continuous approximations of $\omega$ converging to $\omega$ where $\omega$ satisfies Remark \ref{sep}.

First we formulate a general uniqueness theorem.

\begin{theorem}\label{t2}
 Suppose that there are two mild solutions  $u_1,\,u_2$ of \eqref{eq7} with $u_1(0)=u_2(0)=u_0\in V$ and driven by the same path $\omega$. Then, under the before mentioned assumptions on $A$, $B$ and $G$ we have $u_1(t)=u_2(t)$ for $t\in [0,T]$.
\end{theorem}

\begin{proof}
Assume that there exists a maximal interval $[0,t_0]$ contained in $[0,T]$ such that $\Delta u:=u_1-u_2$ is zero on this interval being $t_0<T$. Then there exists a $0<\mu<1$ such that
$\Delta u\not=0$ on $(t_0,t_0+\mu]$.

We divide the proof in several steps:

(i) First we want to estimate
$$|||\Delta u|||_{\beta,-\delta,t_0,t_0+\mu}=\sup_{t_0\leq s<t\leq t_0+\mu}\frac{\|\Delta u (t)-\Delta u(s)\|_{V_{-\delta}}}{(t-s)^\beta}.$$

Regarding the non-stochastic integral, we have to estimate
\begin{align*}
\frac{1}{(t-s)^\beta}&\bigg\|\int_{s}^{t}S(t-r)A^{-\delta}(B(u_1(r),u_1(r))-B(u_2(r),u_2(r)))dr\bigg\|\\
&+\frac{1}{(t-s)^\beta}\bigg\|\int_{t_0}^{s}(S(t-r)-S(s-r))A^{-\delta}(B(u_1(r),u_1(r))-B(u_2(r),u_2(r)))dr\bigg\|\\
&=: J_1+J_2.
\end{align*}

Since $V_{-1/2} \subset V_{-\delta}$, from Lemma \ref{Bgenerale} we obtain
\begin{align*}
  \|A^{-\delta}(B(u_1(r),u_1(r))-B(u_2(r),u_2(r)))\| & \le c\|B(\Delta u(r),u_1(r))\|_{V_{-1/2}}+c \|B(u_2(r),\Delta u(r))\|_{V_{-1/2}}\\
  &\le c\|\Delta u(r)\|(\|u_1(r)\|+\|u_2(r)\|).
\end{align*}
Therefore
\begin{align*}
&J_1\leq \frac{c}{(t-s)^\beta} \int_{s}^{t} \|\Delta u(r)\|(\|u_1(r)\|+\|u_2(r)\|)dr
\leq c \mu^{1-\beta} ||\Delta u||_{C,t_0,t_0+\mu}  (||u_1||_{C,t_0,t_0+\mu}+||u_2||_{C,t_0,t_0+\mu}).
\end{align*}
Notice also that using the properties of the semigroup $S$
\begin{align*}
  \|(S(t-r)&-S(s-r))A^{-\delta}(B(u_1(r),u_1(r))-B(u_2(r),u_2(r)))\| \\
  =&  \|(S(t-s)-{\rm id}) S(s-r)(B(u_1(r),u_1(r))-B(u_2(r),u_2(r)))\|_{V_{-\delta}} \\
   \le & c(t-s)^\delta \|(B(u_1(r),u_1(r))-B(u_2(r),u_2(r)))\| \\
    \le & c (t-s)^\delta \|\Delta u(r)\|(\|u_1(r)\|_{V_{1/2}}+\|u_2(r)\|_{V_{1/2}}),
\end{align*}
 and thus
\begin{align*}
J_2\leq &\frac{1}{(t-s)^\beta}\int_{t_0}^{s} (t-s)^\delta \|\Delta u(r)\| (\|u_1(r)\|_{V_{1/2}}+\|u_2(r)\|_{V_{1/2}}) dr\\
 \leq & \mu^{\frac12+\delta-\beta} \|\Delta u\|_{C,t_0,t_0+\mu}(||u_1||_{L^2(0,T, V_{1/2})}+||u_2||_{L^2(0,T, V_{1/2})}).
\end{align*}
To analyze the terms corresponding to the stochastic integral, that is,
\begin{equation*}
 \sup_{t_0\leq s<t\leq t_0+\mu}\frac{\bigg\| \displaystyle {\int_{s}^{t} S(t-r)(G(u_1(r))-G(u_2(r)))d\omega- \int_{t_0}^{s} (S(t-r)-S(s-r))(G(u_1(r))-G(u_2(r)))d\omega}\bigg\|_{V_{-\delta}}}{(t-s)^\beta}
\end{equation*}we can consider the estimates of $I_4,\,I_5$ given in the Appendix, replacing $\|A^{-{\delta}}(G(u(r))\|_{L_2(V)}$ by
\begin{equation}\label{eq16}
  \|A^{-\delta}(G(u_1(r))-G(u_2(r)))\|_{L_2(V)}\le c_{DG}\|\Delta u(r)\|_{V_{-\delta}}
\end{equation}
and $\|A^{-\delta}(G(u(r))-G(u(q)))\|_{L_2(V)}$ by
\begin{align}\label{eq17}
\begin{split}
\|A^{-\delta}(G(u_1(r))&-G(u_2(r))-(G(u_1(q))-G(u_2(q))))\|_{L_2(V)}\\
&\le
  c_{DG}\|\Delta u(r)-\Delta u(q)\|_{V_{-\delta}}+c_{D^2G}(\|\Delta u(r)\|_{V_{-\delta}}(\|u_1(r)-u_1(q)\|_{V_{-\delta}}\\
  &+\|u_2(r)-u_2(q)\|_{V_{-\delta}}),
\end{split}
\end{align}
where these two above estimates follow by {\bf (G)}. Then following the steps of Lemma \ref{l4} and taking into account
\begin{align*}
  \|\Delta u(r)\|_{V_{-\delta}}= \|\Delta u(r)-\Delta u(t_0)\|_{V_{-\delta}}\le |||\Delta u|||_{\beta,-\delta,t_0,t_0+\mu}(r-t_0)^\beta,
\end{align*}
which is true due to the fact that $\Delta u(t_0)=0$, we obtain the following term as an upper bound of the stochastic part:
$$c|||\omega|||_{\beta^\prime}\mu^{\beta^\prime}|||\Delta u|||_{\beta,-\delta,t_0,t_0+\mu}+c|||\omega|||_{\beta^\prime}\mu^{\beta^\prime+\beta}(|||u_1|||_{\beta,-\delta,0,T}+|||u_2|||_{\beta,-\delta,0,T}) \|\Delta u\|_{C,t_0,t_0+\mu}.$$

Collecting everything we get
\begin{align}\label{esti1}
|||\Delta u|||_{\beta,-\delta,t_0,t_0+\mu} \leq c_\mu^1 |||\Delta u|||_{\beta,-\delta,t_0,t_0+\mu}+ c_\mu^2 \|\Delta u\|_{C,t_0,t_0+\mu},
\end{align}
with
\begin{align}\label{esti2}
\begin{split}
c_\mu^1 &=c \mu^{\beta^\prime} |||\omega|||_{\beta^\prime},
\\
c_\mu^2&=c(\mu^{\frac12+\delta-\beta} (||u_1||_{L^2(0,T, V_{1/2})}+||u_2||_{L^2(0,T, V_{1/2})})+\mu^{\beta^\prime+\beta}|||\omega|||_{\beta^\prime}(|||u_1|||_{\beta,-\delta,0,T}+|||u_2|||_{\beta,-\delta,0,T})\\
&\quad +\mu^{1-\beta} (||u_1||_{C, t_0,t_0+\mu}+||u_2||_{C, t_0,t_0+\mu})).
\end{split}
\end{align}

(ii) In this second step we are interested in estimating $\|\Delta u\|_{C,t_0,t_0+\mu}$. The non-stochastic part gives us
\begin{align*}
&\sup_{t_0\le t \le t_0+\mu}\bigg\|\int_{t_0}^{t}S(t-r)(B(u_1(r),u_1(r))-B(u_2(r),u_2(r)))dr\bigg\|\\
\leq &c \sup_{t_0\le t \le t_0+\mu} \int_{t_0}^{t} \|\Delta u(r)\| (||u_1(r)||_{V_{1/2}}+||u_2(r)||_{V_{1/2}}) dr\\
\leq &c \mu^\frac12 (||u_1||_{L^2(0,T, V_{1/2})}+||u_2||_{L^2(0,T, V_{1/2})}) \|\Delta u\|_{C,-\delta,t_0,t_0+\mu}.
\end{align*}

To study the norm of the stochastic integral, for $t\in [t_0,t_0+\mu]$ we split it as follows
\begin{align*}
&|||\omega|||_{\beta'} \int_{t_0}^{t}(t-r)^{\alpha+\beta'-1} \left(\frac{\|S(t-r)(G(u_1(r))-G(u_2(r)))\|_{L_2(V)}}{(r-t_0)^{\alpha}} \right.\\
&\qquad \qquad \qquad  \left.+\int_{t_0}^{r}\frac{\|(S(t-r)-S(t-\hat r))(G(u_1(r))-G(u_2(r)))\|_{L_2(V)}}{(r-\hat r)^{\alpha+1}}d\hat r \right.
\\&\qquad \qquad \qquad \left.+
\int_{t_0}^{r}\frac{\|S(t-\hat r)((G(u_1(r))-G(u_2(r)))-(G(u_1(\hat r))-G(u_2(\hat r))))\|_{L_2(V)}}{(r-\hat r)^{\alpha+1}}d\hat r \right) dr\\
&\qquad =:J_3(t)+J_4(t)+J_5(t).
\end{align*}
Following the steps of Lemma \ref{l4}, thanks to {\bf(G)} we obtain
\begin{align*}
\sup_{t_0\le t \le t_0+\mu} J_3(t)\leq c |||\omega|||_{\beta'} \mu^{\beta^\prime} \|\Delta u\|_{C,t_0,t_0+\mu},\\
\sup_{t_0\le t \le t_0+\mu} J_4(t) \leq c |||\omega|||_{\beta'} \mu^{\beta^\prime} \|\Delta u\|_{C,t_0,t_0+\mu}.
\end{align*}

Finally, using again {\bf(G)}, since $\|\Delta u(r)\|_{V_{-\delta}}\leq c\|\Delta u(r)\|$,
\begin{align*}
\sup_{t_0\le t \le t_0+\mu} J_5(t) & \leq c |||\omega|||_{\beta'}  \int_{t_0}^{t}(t-r)^{\alpha+\beta'-1} \\
&  \quad  \times \bigg(\int_{t_0}^{r} \frac{\|\Delta u(r)-\Delta u (\hat r)\|_{V_{-\delta}}+ \|\Delta u(r)\| (\|u_1(r)-u_1(\hat r)\|_{V_{-\delta}}+\|u_2(r)-u_2(\hat r)\|_{V_{-\delta}}) }{(r-\hat r)^{\alpha+1}}d\hat r\bigg)  dr\\
&\leq c |||\omega|||_{\beta'} ( |||\Delta u|||_{\beta,-\delta,t_0,t_0+\mu} + \|\Delta u\|_{C,t_0,t_0+\mu} (|||u_1|||_{\beta,-\delta,t_0,t_0+\mu}+|||u_2|||_{\beta,-\delta,t_0,t_0+\mu})  \\
& \quad  \times  \sup_{t_0\le t \le t_0+\mu} \int_{t_0}^{t}(t-r)^{\alpha+\beta'-1} \bigg(\int_{t_0}^{r} (r-\hat r)^{\beta-\alpha-1}d\hat r\bigg)  dr\\
&\leq c |||\omega|||_{\beta'} \mu^{\beta+\beta^\prime}( |||\Delta u|||_{\beta,-\delta,t_0,t_0+\mu} + \|\Delta u\|_{C,t_0,t_0+\mu} (|||u_1|||_{\beta,-\delta,0,T}+|||u_2|||_{\beta,-\delta,0,T})).
\end{align*}
Hence,
\begin{align}\label{esti3}
\|\Delta u\|_{C,t_0,t_0+\mu} \leq c_\mu^3 \|\Delta u\|_{C,t_0,t_0+\mu}+c_\mu^4 |||\Delta u|||_{\beta,-\delta,t_0,t_0+\mu},
\end{align}
with
\begin{align}\label{esti4}
\begin{split}
c_\mu^3 &=c(\mu^\frac12 (||u_1||_{L^2(0,T, V_{1/2})}+||u_2||_{L^2(0,T, V_{1/2})})+\mu^{\beta^\prime} |||\omega|||_{\beta^\prime}\\
&\quad + \mu^{\beta^\prime+\beta}|||\omega|||_{\beta^\prime}(|||u_1|||_{\beta,-\delta,0,T}+|||u_2|||_{\beta,-\delta,0,T})),
\\
c_\mu^4&=c \mu^{\beta^\prime+\beta}|||\omega|||_{\beta^\prime}.
\end{split}
\end{align}

Therefore, solving the system given by (\ref{esti1}) and (\ref{esti3}) means that we have to solve a system of inequalities, namely
$$X\leq c_\mu^1 X+c_\mu^2 Y, \qquad Y\leq c_\mu^3 Y+c_\mu^4 X $$
with $c_\mu^i$ given by (\ref{esti2}) and (\ref{esti4}). It is now straightforward to check that for a small enough $\mu\in (0,1)$ we obtain that $||\Delta u||_{C, t_0,t_0+\mu}=0$, which contradicts the fact that the maximal interval of uniqueness is $[0,t_0]$. Hence the solution of \eqref{abstract} is unique.

\end{proof}

Finally, we can prove the main theorem of the paper:

\begin{theorem}\label{t3}
Under the assumptions of Lemma \ref{l5} there exists a mild solution to the stochastic shell--model \eqref{abstract} with driving function $\omega\in C^{\beta^\prime}([0,T];V)$.
\end{theorem}

\begin{proof} We divide the proof in several steps:

(i) Let $(\omega_n)_{n\in\NN}$ be a sequence of piecewise linear continuous functions converging to $\omega$ in $C^{\beta^\prime}([0,T],V)$, see Remark \ref{sep}, and let $(u_n)_{n\in\N}$ be the sequence of unique solutions driven by $(\omega_n)_{n\in\NN}$ with initial condition $u_0\in V$. From Lemma \ref{corou} we know that $(u_n)_{n\in\N}$ is uniformly bounded in $C^{\hat\beta}([0,T],V_{-\delta})\cap C([0,T],V)$. Then (\ref{E1}) implies that $(\|u_n\|_{L^2(0,T,V_{1/2})})_{n\in\N}$ is also bounded and hence $(u_n)_{n\in\NN}$ is relatively weak compact in $L^2(0,T,V_{1/2})$. Furthermore, this sequence is relatively compact in $L^2(0,T,V)\cap C([0,T],V_{-\delta})$ by Theorem \ref{t1} (i). Moreover, from Lemma \ref{l7} and Lemma \ref{l4} we obtain that
\begin{align*}
|||u_n|||_{\hat\beta,-\hat \delta,0,t}&\leq ct^{\hat \delta-\hat\beta}\|u_n(0)\|+ct^{1-\hat\beta}  \|u_n\|_{C,0,t}^2 +c |||\omega_n|||_{\beta'}t^{\beta^\prime-\hat\beta-\varepsilon} ( 1+t^{\hat\beta}|||u_n|||_{\hat\beta,-\delta,0,t} ),
\end{align*}
which, together with the fact that $(\omega_n)_{n\in \mathbb N}$ converges to $\omega$, imply that $(u_n)_{n\in\NN}$ is uniformly
bounded in $C^{\hat\beta}([0,T],V_{-\hat\delta})$ with $\hat\delta=\delta-\varepsilon$ being $\varepsilon>0$ arbitrarily small such that
$\delta,\,\hat\delta$ satisfies still the conditions of Lemma \ref{l5}. Hence, by Theorem \ref{t1} (ii), this sequence in relatively compact in $C^{\beta}([0,T],V_{-\delta})$.

\smallskip

(ii) Let $(u_{n^\prime})_{n^\prime \in\NN}$ be a subsequence converging to some limit point $u\in L^ 2(0,T,V_{1/2})\cap C^\beta([0,T],V_{-\delta})$. Let us denote this subsequence simply by $(u_n)_{n\in\NN}$. Then, since $B:V_\frac12 \times V_{-\delta} \to V_{-\delta}$ and also $B: V_{-\delta} \times V_\frac12 \to V_{-\delta}$ and $u_n(0)-u(0)=0$, applying Lemma \ref{Bgenerale} we have
\begin{align*}
    &\bigg\|\int_0^tS(t-r)(B(u_n(r),u_n(r))-B(u(r),u(r)))dr\bigg\|_{V_{-\delta}} \\
    &\le \int_0^t (\|B(u_n(r),u_n(r))-B(u(r),u_n(r))\|_{V_{-\delta}}+\|B(u(r),u_n(r))-B(u(r),u(r))\|_{V_{-\delta}})dr \\
      &\le c\int_0^t (\|u_n(r)\|_{V_{1/2}}+\|u(r)\|_{V_{1/2}})\|u(r)-u_n(r)\|_{V_{-\delta}}dr\\
    &\le c |||u-u_n|||_{\beta,-\delta,0,T}\int_0^t r^\beta (\|u_n(r)\|_{V_{1/2}}+\|u(r)\|_{V_{1/2}})dr\\
    &\le c T^{\beta+\frac12} |||u-u_n|||_{\beta,-\delta,0,T}(\|u_n\|_{L^2(0,T,V_{1/2})}+\|u\|_{L^2(0,T,V_{1/2})})
\end{align*}
which shows the convergence in $V_{-\delta}$ of the left hand side to zero.

For the stochastic integral we consider the splitting
\begin{align*}
  \bigg\|&\int_0^tS(t-r)G(u_n(r))d\omega_n(r)-\int_0^tS(t-r)G(u(r))d\omega(r)\bigg\|_{V_{-\delta}}\\
  &\le \bigg\|\int_0^tS(t-r)G(u_n(r))d(\omega_n(r)-\omega(r))\bigg\|_{{V_{-\delta}}}+
  \bigg\|\int_0^tS(t-r)(G(u_n(r))-G(u(r)))d\omega_n(r)\bigg\|_{V_{-\delta}}.
\end{align*}
Similar to  \eqref{eq10}, an upper bound for the first integral on the right hand side is given by
\begin{equation*}
  CT^{\beta'}|||\omega_n-\omega|||_{\beta'}\left(1+T^\beta|||u_n|||_{\beta,-\delta,0,T}\right)
\end{equation*}
and since the set $\{|||u_n|||_{\beta,-\delta,0,T}\}_{n\in\NN}$ is bounded, we obtain the convergence in $V_{-\delta}$ of the first integral on the right hand side.
Now using \eqref{eq16}-\eqref{eq17}, setting $u_1=u_{n},\,u_2=u$ we arrive at
\begin{align*}
 \bigg\|\int_0^tS(t-r)(G(u_n(r))-G(u(r)))d\omega_n(r)\bigg\|_{V_{-\delta}}&\leq  c|||\omega_n|||_{\beta^\prime}T^{\beta^\prime}|||u_n-u|||_{\beta,-\delta,0,T}\\
 \quad &\times (1+T^\beta(1+|||u_n|||_{\beta,-\delta,0,T}+|||u|||_{\beta,-\delta,0,T}))
 \end{align*}
 which shows the convergence  in $V_{-\delta}$ of the second integral.

 Also, since $(u_n)_{n\in \mathbb N}$ converges to $u$ in $C([0,T],V_{-\delta})$, for every $t\in [0,T]$ we have that $u_n(t) \to u(t)$ in $V_{-\delta}$.
 \smallskip

 (iii) Since $u\in L^2(0,T,V_{1/2})\cap L^\infty(0,T,V)$ we have that $t\mapsto B(u(t),u(t))\in L^ 2(0,T,V)$ and hence  the continuity in $V$ of the first integral of
 \eqref{eq7} with respect to $t$ follows. Moreover, since $u\in C^\beta([0,T],V_{-\delta})$ by {\bf (G)} we obtain that
 \begin{equation*}
   t\mapsto \int_0^tS(t-r)G(u(r))d\omega\in C([0,T],V).
 \end{equation*}

(iv) Collecting the above properties, on the one hand (i)-(ii) mean that $u\in C^{\beta}([0,t],V_{-\delta}) \cap L^2(0,T,V_{1/2})$ and $u$ satisfies (\ref{eq7}) in $V_{-\delta}$. On the other hand, (iii) means that the right hand side of (\ref{eq7}) belongs to $C([0,T],V)$, and hence also the left hand side. In conclusion, we have proven the existence of a mild solution $u$ to the stochastic shell--model in the sense of Definition \ref{def1}.
\end{proof}

\section{An example of diffusion term}
We define the operator $G$ by a sequence of functions $g_m^n(u)\in \CC$ with $u\in V_{-\delta}$, such that for $v\in V$:
\begin{equation}\label{ex}
(G(u)v)_{n}:=\sum_{n, m=1}^{\infty}g^{n}_{m}(u)v_{m}\in V.
\end{equation}
We now define properties for this sequence such that $G$ satisfies the hypotheses {\bf (G)}.
For every $n, m=1,\dots $, assume that
\begin{equation}\label{g}
\sup_{u\in V_{-\delta}}\sum_{n, m=1}^{\infty}|g_{m}^{n}(u)|^{2}=:c_G^2<\infty .
\end{equation}
In addition, let us assume that the operators $g_m^n$ are twice differentiable having the following properties:
For $u, h\in V_{-\delta}$ and $(f_k)_{k\in\NN}$ an orthonormal base in $V_{-\delta}$ we have that

\begin{align}\label{g2}
\begin{split}
&\sum_{n, m=1}^{\infty}\left(g_{m}^{n}(u+h)-g_{m}^{n}(u)-Dg_{m}^{n}(u)h\right)^2=
\sum_{n, m=1}^{\infty}\left(o^{n,m}_u(\|h\|_{V_{-\delta}})\right)^2=o_u(\|h\|_{V_{-\delta}})^2,\\
&\sup_{u\in V_{-\delta}}\sum_{n, m,k=1}^{\infty}|D g_{m}^{n}(u)f_k|^{2}=:c_{DG}^2<\infty.
\end{split}
\end{align}
The $o^{n,m}_u,\,o_u$ have the usual properties: $\lim_{h\to 0}|o^{n,m}_u(\|h\|_{V_{-\delta}})|/\|h\|_{V_{-\delta}}=0$  and similar for $o_u$.
In addition we assume that for $u, h_1, h_2\in V_{-\delta}$

\begin{align}\label{g3}
\begin{split}
&\sum_{n, m=1}^{\infty}\left(Dg_{m}^{n}(u+h_2)h_1-Dg_{m}^{n}(u)h_1-D^2g_{m}^{n}(u)h_1 h_2\right)^2=
\sum_{n, m=1}^{\infty}\left(o^{n,m}_{u, h_1}(\|h_2\|_{V_{-\delta}})\right)^2
=:\left(o_{u,h_1}(\|h_2\|_{V_{-\delta}})\right)^2,\\
&\sup_{u\in V_{-\delta}}
\sum_{n, m,k,l=1}^{\infty}|D^2 g_{m}^{n}(u)(f_k,f_l)|^{2}=:c_{D^2G}^2<\infty
\end{split}
\end{align}
where the {\em little o's} have the same property as above.

\smallskip

Now we can verify the properties of the operator $G$  formulated in hypothesis {\bf(G)}.
It follows from \eqref{g} that
\begin{align*}
 \sup_{u\in V_{-\delta}}\|G(u)\|_{L_2(V)}^2&=\sup_{u\in V_{-\delta}}\sum_{m=1}^{\infty} \|G(u)e_{m}\|^2=
 \sup_{u\in V_{-\delta}}\sum_{n,m=1}^{\infty} |(G(u)e_{m})_{n}|^2\\
 & = \sup_{u\in V_{-\delta}}\sum_{n,m=1}^{\infty}|g^{n}_{m}(u)|^2=c_G^2.
\end{align*}

Simple calculations show that \eqref{g2}, \eqref{g3} imply that the operator $DG$ and $D^2G$ exist and are bounded. In fact, if $u, h\in V_{-\delta}$, then we have that
\begin{equation*}
\|G(u+h)-G(u)-DG(u)h\|_{L_{2}(V)}^{2}=\left(o(\|h\|_{V_{-\delta}})\right)^2
\end{equation*}
and
\begin{align*}
 \sup_{u\in V_{-\delta}}\|DG(u)\|_{L_2(V \times V_{-\delta},V)}^2 &=  \sup_{u\in V_{-\delta}}\sum_{m,k=1}^{\infty} \|DG(u)(e_m,f_k)\|^2= \sup_{u\in V_{-\delta}}\sum_{n,m,k=1}^{\infty}
 |Dg^{n}_{m}(u)f_k|^2 =c_{DG}^2.
 \end{align*}

Now, using the boundedness of $DG$ we can prove the Lipschitz condition.
 Similarly, \eqref{g3} implies that the operator $D^2G$ exists and is bounded.  Using the boundedness of the second derivative of $G$ standard calculations give \eqref{eq12}.

\section{Appendix}

We start this section by completing the proof of Lemma \ref{l7}, although in the next result (item (i)) we prove a bit more.

\begin{lemma}\label{l4}
(i) Let $I_4,\,I_5$ be defined in \eqref{eq4}. Then for any sufficient small $\varepsilon\geq  0$ such that $\beta^\prime-\hat \beta>\varepsilon$ we have
\begin{equation*}
\sup_{0\leq p<q\leq t} \frac{\|A^\varepsilon I_4(p,q)\|+\|A^\varepsilon I_{5}(p,q)\|}{(q-p)^{\hat\beta}} \leq c t^{\beta'-\hat\beta-\varepsilon}|||\omega|||_{\beta'}
\left(1+t^{\hat\beta}|||u|||_{\hat\beta,-\delta,0,t}\right).
\end{equation*}

\smallskip

(ii) Let $I_2,\, I_3$ be defined in  \eqref{eq4} and let $0\le \varepsilon<\delta-1/2$.
Then
\begin{equation*}
  \sup_{0\leq p<q\leq t} \frac{\|A^\varepsilon I_{2}(p,q)\|+\|A^\varepsilon I_{3}(p,q)\|}{(q-p)^{\hat\beta}}\le ct^{1-\hat \beta}\|u\|_{C,0,t}^2
\end{equation*}
\end{lemma}
Note that \eqref{eq10} follows then by (i) simply taking $\varepsilon=0$.
\begin{proof}
Throughout the proof we will use frequently the properties \eqref{eq1}, \eqref{eq2} and (\ref{prop}).  We choose an $\alpha$ in the same conditions than in Lemma \ref{l5}, that is, $1-\beta^\prime<\alpha<\hat \beta$.

First, using the definition of the stochastic integral and the estimate (\ref{er}),
\begin{align*}
&\sup_{0\leq p<q\leq t}\frac{\|A^\varepsilon I_{4}(p,q)\|}{(q-p)^{\hat\beta}}\leq \sup_{0\leq p<q\leq t} \frac{1}{(q-p)^{\hat\beta}}
 |||\omega|||_{\beta'} \int_{p}^{q}(q-r)^{\alpha+\beta'-1} \left(\frac{\|S(q-r)A^\varepsilon A^{-\delta}G(u(r))\|_{L_2(V)}}{(r-p)^{\alpha}} \right.\\
&\qquad \left.+\int_{p}^{r}\frac{\|(S(q-r)-S(q-\hat r))A^\varepsilon A^{-\delta}G(u(r))\|_{L_2(V)}}{(r-\hat r)^{\alpha+1}}d\hat r+
\int_{p}^{r}\frac{\|S(q-\hat r)A^\varepsilon A^{-\delta}(G(u(r))-G(u(\hat r)))\|_{L_2(V)}}{(r-\hat r)^{\alpha+1}}d\hat r \right) dr.
\end{align*}

The first term is estimated by
\begin{align*}
\frac{\|S(q-r)A^\varepsilon A^{-\delta}G(u(r))\|_{L_2(V)}}{(r-p)^{\alpha}}
&\leq c\frac{c_G}{(r-p)^{\alpha}(q-r)^\varepsilon}
\end{align*}

 and since $\alpha+\beta^\prime-\varepsilon>0$, we get
\begin{align*}
\sup_{0\leq p<q\leq t} c \frac{c_G}{(q-p)^{\hat\beta}}
 |||\omega|||_{\beta'} \int_{p}^{q} (q-r)^{\alpha+\beta'-\varepsilon-1}(r-p)^{-\alpha} dr \leq c |||\omega|||_{\beta'} t^{\beta'-\hat\beta-\varepsilon}.
\end{align*}

Concerning the second term, taking an appropriate $\alpha^\prime >\alpha$  such that $\alpha+\beta^\prime> \alpha^\prime+\varepsilon$, we have
\begin{align*}
\int_{p}^{r}&\frac{\|(S(q-r)-S(q-\hat r))A^\varepsilon A^{-\delta}G(u(r))\|_{L_2(V)}}{(r-\hat r)^{\alpha+1}}d\hat r
\leq \frac{c\,c_G}{(q-r)^{\alpha'+\varepsilon}}\int_{p}^{r}\frac{(r-\hat r)^{\alpha'}}{(r-\hat r)^{\alpha+1}}d\hat r\leq \frac{c\, c_G(r-p)^{\alpha'-\alpha}}{(q-r)^{\alpha'+\varepsilon}},
\end{align*}
and hence
\begin{align*}
\sup_{0\leq p<q\leq t}  \frac{c\, c_G}{(q-p)^{\hat\beta}}
 |||\omega|||_{\beta'} \int_p^q\frac{(r-p)^{\alpha'-\alpha}}{(q-r)^{\alpha'+\varepsilon}}(q-r)^{\alpha+\beta^\prime-1}dr\le c |||\omega|||_{\beta'}t^{\beta^\prime-\hat\beta-\varepsilon}.
\end{align*}

Finally, since $\hat\beta>\alpha$

\begin{align*}
\int_{p}^{r}&\frac{\|S(q-\hat r)A^\varepsilon A^{-\delta}(G(u(r))-G(u(\hat r)))\|_{L_2(V)}}{(r-\hat r)^{\alpha+1}}d\hat r\leq
c\int_{p}^{r}\frac{\|A^{-\delta}(G(u(r))-G(u(\hat r)))\|_{L_2(V)}}{(r-\hat r)^{\alpha+1}(q-\hat r)^\varepsilon}d\hat r\\
&\leq c c_{DG} |||u|||_{\hat\beta,-\delta,0,t}\frac{1}{(q-r)^\varepsilon}\int_{p}^{r}\frac{(r-\hat r)^{\hat \beta}}{(r-\hat r)^{\alpha+1}}d\hat r
\leq c c_{DG} |||u|||_{\hat\beta,-\delta,0,t} \frac{(r-p)^{\hat\beta-\alpha}}{(q- r)^\varepsilon},
\end{align*}

and since $\beta'+\alpha-\varepsilon >0$  we have
\begin{align*}
\sup_{0\leq p<q\leq t}  \frac{c c_{DG} |||u|||_{\hat\beta,-\delta,0,t}}{(q-p)^{\hat\beta}}
 |||\omega|||_{\beta'}  \int_{p}^{q}(q-r)^{\alpha+\beta'-\varepsilon-1}(r-p)^{\hat\beta-\alpha}dr& \leq c |||\omega|||_{\beta'} |||u|||_{\hat\beta,-\delta,0,t} t^{\beta'-\varepsilon}.
\end{align*}

Hence, we get that
\begin{align*}
\sup_{0\leq p<q\leq t} & \frac{\|A^{\varepsilon}I_{4}(p,q)\|}{(q-p)^{\hat\beta}}\leq
ct^{\beta'-\hat\beta-\varepsilon}|||\omega|||_{\beta'}
\left(1+t^{\hat\beta}|||u|||_{\hat\beta,-\delta,0,t}\right).
\end{align*}

Thanks to the definition of the stochastic integral and the estimate (\ref{er}) for $I_{5}$ we get
\begin{align*}
\sup_{0\leq p<q\leq t} & \frac{\|A^{\varepsilon}I_{5}(p,q)\|}{(q-p)^{\hat\beta}}\leq
\sup_{0\leq p<q\leq t} \frac{1}{(q-p)^{\hat\beta}} |||\omega|||_{\beta'} \int_{0}^{p}(p-r)^{\alpha+\beta'-1} \bigg( \frac{\|(S(q-r)-S(p-r))A^{\varepsilon}A^{-\delta}G(u(r))\|_{L_2(V)}}{r^{\alpha}} \\
&\qquad \qquad +\int_{0}^{r}\frac{\|(S(q-\hat r)-S(p-\hat r))A^{\varepsilon}A^{-\delta}(G(u(r))-G(u(\hat r)))\|_{L_2(V)}}{(r-\hat r)^{\alpha+1}}d\hat r \\
&\qquad \qquad +\int_{0}^{r}\frac{\|(S(q-r)-S(q-\hat r)-S(p-r)+S(p-\hat r))A^{\varepsilon}A^{-\delta}G(u(r))\|_{L_2(V)}}{(r-\hat r)^{\alpha+1}} d\hat r\bigg) dr \\
&=:\sup_{0\leq p<q\leq t}\frac{1}{(q-p)^{\hat\beta}} |||\omega|||_{\beta'} \int_{0}^{p}(p-r)^{\alpha+\beta'-1}
\left(I_{5,1}+I_{5,2}+I_{5,3}\right)dr.
\end{align*}

We start with
\begin{align*}
I_{5,1}&=\frac{\|(S(q-p)-{\rm Id})S(p-r)A^{\varepsilon}A^{-\delta}G(u(r))\|_{L_2(V)}}{r^{\alpha}}\leq c \frac{ c_G(q-p)^{\hat\beta}}{(p-r)^{\hat\beta+\varepsilon}r^{\alpha}}
\end{align*}
and because $\alpha<1/2$ and $\alpha+\beta^\prime-\hat\beta-\varepsilon>0$, the term involving $I_{5,1}$ is estimated by
\begin{align*}
\sup_{0\leq p<q\leq t}\frac{c \, c_G }{(q-p)^{\hat\beta}} |||\omega|||_{\beta'}(q-p)^{\hat\beta}\int_{0}^{p}(p-r)^{\alpha+\beta'-1-\hat\beta-\varepsilon}r^{-\alpha}dr
&\leq  c |||\omega|||_{\beta'} t^{\beta'-\hat\beta-\varepsilon}.
\end{align*}

On the other hand, \begin{align*}
I_{5,2}&= \int_{0}^{r} \frac{\|(S(q-p)-{\rm Id})S(p-\hat r)A^{\varepsilon}A^{-\delta}(G(u(r))-G(u(\hat r)))\|_{L_2(V)}}{(r-\hat r)^{\alpha+1}}d\hat r\\
&\leq c\, c_{DG}\int_{0}^{r}\frac{(p-\hat r)^{-\hat\beta-\varepsilon} (q-p)^{\hat\beta} \|u(r)-u(\hat r)\|_{V_{-\delta}}}{(r-\hat r)^{\alpha+1}}d\hat r\\
&\leq  c\, c_{DG} |||u|||_{\hat\beta,-\delta,0,t}(p-r)^{-\hat\beta-\varepsilon} (q-p)^{\hat\beta} \int_{0}^{r}  \frac{1}{(r-\hat r)^{\alpha+1-\hat\beta}}d\hat r\\
&\leq   c\, c_{DG} |||u|||_{\hat\beta,-\delta,0,t} (p-r)^{-\hat\beta-\varepsilon} (q-p)^{\hat\beta} r^{\hat \beta-\alpha},
\end{align*}
and thus
\begin{align*}
\begin{split}
\sup_{0\leq p<q\leq t} &  \frac{1}{(q-p)^{\hat\beta}} |||\omega|||_{\beta^\prime} \int_{0}^{p}(p-r)^{\alpha+\beta'-1}I_{5,2}dr\\
&\leq  c\, c_{DG}  |||\omega|||_{\beta^\prime} |||u|||_{\hat\beta,-\delta,0,t}  \sup_{0\leq p<q\leq t}   \int_{0}^{p} (p-r)^{\alpha+\beta'-1-\hat\beta-\varepsilon}
r^{\hat \beta -\alpha}dr\\
&\leq c\, c_{DG}  |||\omega|||_{\beta^\prime} |||u|||_{\hat\beta,-\delta,0,t}  t^{\beta'-\varepsilon}.
\end{split}
\end{align*}

Finally, taking $\alpha^\prime$ close enough to $\alpha$ such that $\alpha^\prime >\alpha$ and $\alpha+\beta^\prime> \alpha^\prime+\hat\beta+\varepsilon$ (for a small enough $\varepsilon$), applying the second part of \eqref{eq30}
\begin{align*}
I_{5,3}&\leq c\int_{0}^{r}\frac{(q-p)^{\hat\beta}(r-\hat r)^{\alpha'}(p-r)^{-\alpha'-\hat\beta-\varepsilon}\|A^{-\delta}G(u(r))\|_{L_2(V)}}
{(r-\hat r)^{\alpha+1}}d\hat r\\
&\leq c \,c_G (q-p)^{\hat\beta}(p-r)^{-\alpha'-\hat\beta-\varepsilon}
\int_{0}^{r}(r-\hat r)^{\alpha'-\alpha-1}d\hat r\\
&\leq c \, c_G (q-p)^{\hat\beta}(p-r)^{-\alpha'-\hat\beta-\varepsilon}r^{\alpha'-\alpha},
\end{align*}
and hence
\begin{align*}
\sup_{0\leq p<q\leq t} &  \frac{1}{(q-p)^{\hat\beta}} |||\omega|||_{\beta^\prime} \int_{0}^{p}(p-r)^{\alpha+\beta'-1}I_{5,3}dr\\
&\leq
c \, \tilde c_G |||\omega|||_{\beta^\prime}  \sup_{0\leq p<q\leq t} \int_{0}^{p}  (p-r)^{\alpha+\beta'-1-\alpha'-\hat\beta-\varepsilon}r^{\alpha'-\alpha}dr\\
&\leq c \, \tilde c_G |||\omega|||_{\beta^\prime}  t^{\beta^\prime-\hat\beta-\varepsilon} .
\end{align*}

Taking into account the previous estimates we finally get
\begin{align*}
\sup_{0\leq p<q\leq t} & \frac{\|A^{\varepsilon}I_{5}(p,q)\|}{(q-p)^{\hat\beta}} \leq c t^{\beta'-\hat\beta-\varepsilon}|||\omega|||_{\beta'}
\left(1+t^{\hat\beta}|||u|||_{\hat\beta,-\delta,0,t}\right).
\end{align*}

(ii) The proof of this part follows similarly to the estimates \eqref{eq18} and \eqref{eq19}. In particular, for the estimate of $\|A^\varepsilon I_2(p,q)\|$ we need to use the continuous embedding $V \subset V_{-\delta+\varepsilon+1/2}$, which holds true for small enough $\varepsilon\ge 0$ since $\delta \in (\hat \beta, 1)$.
\end{proof}

The rest of the Appendix section is devoted to the proof of Lemma \ref{corou}, which relies upon several results that are proven below.


\begin{lemma}\label{l6}
Let $1/2<\hat \beta<\tilde \beta<\delta$ and suppose that $u\in C^{\tilde \beta}([0,T],V_{-\delta})$. Then the mapping
\begin{equation*}
 [s,T]\ni  t\mapsto |||u|||_{\hat \beta,-\delta,s,t}
\end{equation*}
is continuous and
\begin{equation*}
  \lim_{t\to s^+} |||u|||_{\hat \beta,-\delta,s,t}=0.
\end{equation*}
\end{lemma}
\begin{proof}
We only consider here the case $s=0$.  Let us define the following transform of $u$ given by
\begin{equation*}
  \hat u_{\hat t}(r)=\left\{\begin{array}{lcr}
  u(r)&:& r\le \hat t,\\
  u(\hat t)&:& r\ge \hat t.
  \end{array}
  \right.
\end{equation*}
Then for $0\le \hat t <t\le T$
\begin{align*}
  &|||u|||_{\hat \beta,-\delta,0,t}-|||u|||_{\hat \beta,-\delta,0,\hat t}=|||u|||_{\hat \beta,-\delta,0,t}-|||\hat u_{\hat t}|||_{\hat \beta,-\delta,0,t}\le|||u|||_{\hat \beta,-\delta,\hat t,t}
  \le c (t-\hat t)^{\tilde \beta-\hat \beta}|||u|||_{\tilde \beta,-\delta,0,T}
\end{align*}
from which the desired continuity follows immediately. The convergence to 0 follows in the same way.
\end{proof}

\begin{lemma}\label{xx1}
For positive continuous functions $a(t),\,b(t)$ consider
\begin{equation*}
  Y=b(t)+a(t) Y^2
\end{equation*}
and assume $4a(t)b(t)<1$ for every $t\in [0,t_1]$, where $t_1>0$ is some positive number. Then there exist two real solutions $Y_1(t)<Y_2(t)\in\RR^+$ given by
\begin{equation*}
  Y_{1}(t)=\frac{1}{2a(t)}(1-\sqrt{1-4a(t)b(t)}),\quad   Y_{2}(t)=\frac{1}{2a(t)}(1+\sqrt{1-4a(t)b(t)})
\end{equation*}
where $Y_1(t)\le 2 b(t)$.
Suppose in addition that $y(t)\ge0$ is continuous on $[0,t_1]$ such that
\begin{equation*}
   y(t)\le b(t)+a(t) y(t)^2,\quad \lim_{t\to 0^+}y(t)=0,
\end{equation*}
and that $\lim_{t\to 0^+}a(t)=0$.
Then we have $y(t)\le Y_1(t)$ on $[0,t_1]$.
\end{lemma}
\begin{proof}
It follows by Sohr \cite{Sohr} Page 317 that under the conditions of the lemma there exist real solutions $Y_1$, $Y_{2}$ satisfying the above conditions.

On the other hand, $y$ satisfies  the above inequality if and only if $y(t)\le Y_1(t)$ or $y(t)\ge Y_2(t)$. If $y(t)\ge Y_2(t)$ for some $t\in (0,t_1]$ then by the continuity of $y,\,Y_2$ and by the fact that $Y_2(t)>Y_1(t)$ on $(0,t_1]$, it follows that $y(t)\ge Y_2(t)$ on $[0,t_1]$. However, under the assumptions we have $\lim_{t\to 0^+}Y_2(t)=+\infty$ and this is a contradiction with respect to the behavior of $y$.
\end{proof}

To simplify the presentation of the following technical result we assume that $T=1$.
In the following, see Lemma \ref{l8} below, we shall consider inequalities of the type
\begin{equation}\label{ec}
  y(t)\le d(t,x)y(t) +f(t,x)+h(t) y(t)^2,\quad t\in [0,t_1]
\end{equation}
where the increasing functions $d(\cdot,x),\,f(\cdot,x),\,h(\cdot)$ are defined by
\begin{align}\label{ec1}
&d(t,x)=ct^{\beta^\prime}+{ 2}c^3t^{1+2\beta^\prime}+{ 2} c^2xt^{1+\beta^\prime},\nonumber\\
&f(t,x)=x t^{\delta-\hat\beta}
+cx^2t^{1-\hat\beta}+c^2xt^{1+\beta^\prime-\hat\beta}+c^3t^{1+2\beta^\prime-\hat\beta}+ct^{\beta^\prime-\hat\beta},\\
&h(t)={ 4} c^3t^{1+2\beta^\prime+\hat\beta}.\nonumber
\end{align}
Note that $d(t,x)$ and $f(t,x)$ depend on a positive parameter $x$. Furthermore, the constants $c,\,c^2,\,c^3$ are coming from the estimates of Lemma \ref{l5} and Lemma \ref{l7}, as we will  show in Lemma \ref{l8} below. In particular, these constants are constricted such that they are including the value $|||\omega|||_{\beta^\prime}=|||\omega|||_{\beta^\prime,0,1}$. In the following proof we need these constants with norms only for subintervals of $[0,1]$. However,
using $|||\omega|||_{\beta^\prime,0,1}$ these constants can be chosen independently of the subinterval.
In that result, depending on the value of $x$ we shall choose $t_1>0$ such that $d(t_1,x)\le 1/2$. Then, defining $a(t):=2h(t)$ and $b(t,x):=2f(t,x)$ we can rewrite (\ref{ec}) as
\begin{equation*}
  y(t)\le b(t,x)+a(t) y(t)^2,\quad t\in [0,t_1]
\end{equation*}
which looks like the inequality of Lemma \ref{xx1}. Let us emphasize that with the above choice $\lim_{t\to 0^+}a(t)=\lim_{t\to 0^+}2h(t)=0$. In the next result we will also choose suitable values of $x$ such that the rest of assumptions of Lemma \ref{xx1} also hold.\\

\begin{lemma}\label{l8}
Let $u$ be a solution of \eqref{abstract} on $[0,1]$ with initial condition $u_0\in V$ and driven by a piecewise linear continuous path $\omega$. Then for any $x_0\ge \max\{1, \bar c\|u_0\|\ , \|u_0\|\} $ (where $\bar c$ here denotes the constant of (\ref{beta})) there exist constants $K\ge \hat K>1$ defining finitely many intervals $(I_i)_{i=1,\cdots,i^\ast}$ by
\begin{equation*}
  I_1=[0,\frac{1}{\hat K}]=[\check t_1,\hat t_1],\cdots, I_i=[\hat t_{i-1},\hat t_{i-1}+\frac{1}{Ki}]=[\check t_i,\hat t_i]
\end{equation*}
in such a way that on $I_i$ we have
\begin{align*}
   |||u|||_{\hat\beta,-\delta,I_1}\le (\hat K)^{\hat \beta},\quad  |||u|||_{\hat\beta,-\delta,I_i} \le(Ki)^{\hat\beta}\\
  \|u\|_{C,I_1}\le \frac{3c {\hat K}^{1-\beta^\prime}}{1-\beta^\prime},\quad \|u\|_{C,I_i}\le \frac{3c (Ki)^{1-\beta^\prime}}{1-\beta^\prime}
\end{align*}
for $i=2,\cdots, i^\ast$. This constant $c$ in particular depends on $|||\omega|||_{\beta^\prime,0,1}$.

\end{lemma}
We point out that in the previous result $i^\ast$ is given by the condition $\check t_{i^\ast}<1=T\le \hat t_{i^\ast}$, and in this case we set $\hat t_{i^\ast}=1$. 

\begin{proof}
We abbreviate $x_1(t):=\max \{1,\|u\|_{C,0,t}\}$ and $y_1(t):=|||u|||_{\hat\beta,-\delta,0,t}$, for $t\in I_1=[0,\hat t_1]$, where $\hat t_1$ will be determined later.
The inequality  \eqref{E1} together with the fact that $x_0\ge \max\{1, \|u_0\|\}$ imply
\begin{align*}
\|u(t)\|^{2}
& \leq \|u_{0}\|^{2}+c t^{\beta^\prime}\|u\|_{C,0,t}+c t^{\hat{\beta}+\beta^\prime}(1+\|u\|_{C,0,t})|||u|||_{\hat \beta,-\delta,0,t}\\
&\leq x_0^{2}+c t^{\beta^\prime}\|u\|_{C,0,t}+c t^{\hat{\beta}+\beta^\prime}(1+\|u\|_{C,0,t})|||u|||_{\hat \beta,-\delta,0,t},
\end{align*}
and also
\begin{align*}
1
&\leq x_0^{2}+c t^{\beta^\prime}\|u\|_{C,0,t}+c t^{\hat{\beta}+\beta^\prime}(1+\|u\|_{C,0,t})|||u|||_{\hat \beta,-\delta,0,t}.
\end{align*}
Then
\begin{align*}
\max\{\|u\|_{C,0,t}^{2},1\}
&\leq x_0^{2}+c t^{\beta^\prime}\|u\|_{C,0,t}+c t^{\hat{\beta}+\beta^\prime}(1+\|u\|_{C,0,t})|||u|||_{\hat \beta,-\delta,0,t}\\
&\leq x_0^{2}+c t^{\beta^\prime}\max\{1,\|u\|_{C,0,t}\}+c t^{\hat{\beta}+\beta^\prime}(1+\max\{1,\|u\|_{C,0,t}\})|||u|||_{\hat \beta,-\delta,0,t}
\end{align*}
and therefore
\begin{equation}\label{eq21}
  x_1^2(t)\le x_0^2+c\,x_1(t)   t^{\beta^\prime}+2c x_1(t)\,y_1(t)  t^{\hat\beta+\beta^\prime}.
\end{equation}
Furthermore, \eqref{beta} implies
\begin{equation}\label{eq22}
  y_1(t)\le x_0  t^{\delta-\hat\beta}+c x_1^2(t)  t^{1-\hat\beta}+c t^{\beta^\prime-\hat\beta}+cy_1(t) t^{\beta^\prime}.
\end{equation}
Note that in \eqref{eq21} we have used that $x_1(t)\ge 1$ and therefore the corresponding last term on the left hand side of (\ref{E1}) can be estimated as
$$c   t^{\hat\beta+\beta^\prime} (1+x_1(t)) y_1(t) \leq 2c x_1(t)\,y_1 (t)  t^{\hat\beta+\beta^\prime}.$$
Now combining \eqref{eq21} with \eqref{eq22} we get
\begin{equation}\label{eq14}
  y_1(t)\le x_0 t^{\delta-\hat\beta}+c(x_0^2+c\,x_1 (t)  t^{\beta^\prime}+2c x_1(t)\,y_1 (t) t^{\hat\beta+\beta^\prime})  t^{1-\hat\beta}+c t^{\beta^\prime-\hat\beta}+cy_1(t) t^{\beta^\prime}.
\end{equation}
In addition, from \eqref{eq21} the following estimate holds
\begin{equation}\label{eq8}
  x_1(t)\le \frac{c t^{\beta^\prime}+2c y_1(t)  t^{\hat\beta+\beta^\prime}}{2}+\sqrt{\frac{(c  t^{\beta^\prime}+2c y_1(t)  t^{\hat\beta+\beta^\prime})^2+4x_0^2}{4}}
   \le c  t^{\beta^\prime}+2c y_1(t)  t^{\hat\beta+\beta^\prime}+x_0
\end{equation}
and plugging this into \eqref{eq14} we finally arrive at
\begin{equation*}
y_1(t)\le d(t,x_0)y_1(t)+ f(t,x_0)+h(t)y_1(t)^2,\quad t\in I_1=[0,\hat t_1],
\end{equation*}
where the functions have been defined in (\ref{ec1}).
Then, taking $a(t)=2h(t)$ and $b(t,x_0)=2f(t,x_0)$ there exists a $K_1\ge 1$ such that for any $\hat K\ge K_1$ and $\hat t_1=\hat K^{-1}$ we have
\begin{equation*}
  d(\hat K^{-1},x_0)\le \frac12, \quad b(\hat K^{-1},x_0)\le \frac{(\hat K)^{\hat \beta}}{2},\quad 4a(\hat K^{-1})b(\hat K^{-1},x_0)<1.
\end{equation*}
Hence, we have the conditions Lemma \ref{xx1} and as a consequence we claim that $y_1(\hat t_1)\le (\hat K)^{\hat \beta}=\hat t_1^{-\hat\beta}$.
Let us fix such a $\hat K$ such that in addition $x_0\le c\hat K^{1-\beta^\prime}/(1-\beta^\prime)$. Then, from (\ref{eq8}), simply using the general notation for constants $c$, we get
\begin{equation*}
  x_1 (\hat t_1) \le 3c \hat t_1^{\beta^\prime}+x_0\le \frac{4c {\hat K}^{1-\beta^\prime}}{1-\beta^\prime}=:\hat x_1.
\end{equation*}

Now we can repeat the same arguments than above in each interval $I_i$, $i=2,3,\cdots,$ by doing the corresponding suitable changes. In order to do that we need to rewrite the estimates (\ref{E1}) and (\ref{beta}) in those intervals. In particular, in $I_i$ we have to take as initial condition $u(\hat t_{i-1})$ and $t$ can be estimated by the length of the interval $I_i$ which is nothing but $(Ki)^{-1}$. Taking $x_i(t):=\max\{1,\|u\|_{C,\hat t_{i-1},t}\}$ with $x_i(\hat t_{i-1})\ge 1$ and $y_i(t):=|||u|||_{\hat\beta,-\delta,\hat t_{i-1},t}$, for $t\in I_i$. For induction we assume
\begin{equation*}
x_{i-1}(t_{i-1})\le 3c\sum_{j=1}^{i-1}K_j^{-\beta^\prime}+x_0\le   \frac{4c (K(i-1))^{1-{\beta^\prime}}}{1-\beta^\prime}=: \hat x_{i-1}
\end{equation*}

and choose $K>\hat K$ such that for $i=2,3,\cdots$
\begin{align*}
d((Ki)^{-1},\hat x_{i-1})&=c(Ki)^{-\beta^\prime}+{ 2}c^3(Ki)^{-1-2\beta^\prime}+{ 2}c^2\hat x_{i-1}(Ki)^{-\beta^\prime-1}\\
&\le cK^{-\beta^\prime}+{ 2}c^3K^{-1-2\beta^\prime}+\frac{ 8c^3}{1-\beta^\prime}K^{1-\beta^\prime}K^{-\beta^\prime-1}
\le
o(K^{-\varepsilon})\le \frac12\\
f((Ki)^{-1},\hat x_{i-1})&=\hat x_{i-1}(Ki)^{\hat \beta-\delta}+c \hat x_{i-1}^2(Ki)^{\hat\beta-1}+c^2\hat x_{i-1}(Ki)^{-1+\hat\beta-\beta^\prime}+c^3(Ki)^{-1-2\beta^\prime+\hat\beta}+c(Ki)^{-\beta^\prime+\hat\beta} \\
  &\le Co(K^{-\varepsilon}) (Ki)^{\hat\beta} \le\frac{(Ki)^{\hat\beta}}{4}
\end{align*}
for a constant $C$ and an sufficiently small $\varepsilon>0$ independent of $K$ and $i$.
For example, for the critical term in the expression of $f$ given for the quadratic term, we have that
\begin{equation*}
  c \hat x_{i-1}^2(Ki)^{\hat\beta-1}\le \frac{{ 16}c^3}{(1-\beta^\prime)^2} K^{2-2\beta^\prime-1+\hat\beta}i^{2-2\beta^\prime-1+\hat\beta}\le \frac{16 c^3}{(1-\beta^\prime)^2}K^{1-2\beta^\prime}(Ki)^{\hat\beta}\leq Co(K^{-\varepsilon}) (Ki)^{\hat\beta}
\end{equation*}
where this last inequality is true since $\beta^\prime\in (1/2,1)$.

Again, for $a(t)=2h(t),\,b(t,\hat x_{i-1})=2f(t,\hat x_{i-1})$, choosing $K$ sufficiently large such that
\begin{equation*}
4a((Ki)^{-1})b((Ki)^{-1},\hat x_{i-1})\le 16c^3(Ki)^{-1-2\beta^\prime-\hat\beta} \frac{(Ki)^{\hat\beta}}{4}<1
\end{equation*}
we obtain by Lemma \ref{l6} and Lemma \ref{xx1} that $y_i(\hat t_{i})\le (Ki)^{\hat\beta}$. If we denote $\hat t_i-\check t_i=:\Delta t_i=(Ki)^{-1}$ the previous inequality can be rewriten as $y_i(\hat t_{i})\le (\Delta t_i)^{-{\hat \beta}}$, and similar to \eqref{eq8}
\begin{align*}
  x_i (\hat t_i) & \le c \Delta t_i^{\beta^\prime}+2c y_i (\hat t_i)\Delta t_i^{\hat\beta+\beta^\prime}+x_{i-1}(\hat t_{i-1})\leq 3c\Delta t_i^{\beta^\prime}+x_{i-1}(\hat t_{i-1})\\
  &\le x_{0}+3c\sum_{j=1}^i(Kj)^{-\beta^\prime}\le x_0 + 3cK^{-\beta^\prime}\int_0^ir^{-\beta^\prime}dr\le x_0+\frac{3cK^{-\beta^\prime}i^{1-\beta^\prime}}{1-\beta^\prime}\le
 \frac{4c (Ki)^{1-\beta^\prime}}{1-\beta^\prime}=:\hat x_i
\end{align*}
and therefore we obtain that $x_i (\hat t_i)\leq \hat x_i$.
\end{proof}

Finally we present the proof of Lemma \ref{corou}:
\begin{proof}
Consider the sequence $(u_n)_{n\in \NN}$ of weak solutions of (\ref{eq5}) driven by the sequence $(\omega_n)_{n\in \NN}$ of piecewise linear continuous paths. Following the steps of Proposition \ref{prop} we could prove that each $u_n\in C^{\tilde \beta}([0,T];V_{-\delta})$ with $1/2<\hat \beta <\tilde \beta$. Then we can apply Lemmas \ref{l6}-\ref{l8} to $(u_n)_{n\in \NN}$, obtaining that this sequence is uniformly bounded in $C^{\hat\beta}([0,T],V_{-\delta}) \cap C([0,T],V)$.
\end{proof}

\section*{Acknowledgements}
Hakima Bessaih's research was supported in part by the Simons Foundation
grant \#283308 and NSF grant  \#1416689.
H. B.~is thankful to the warm hospitality and the great scientific atmosphere of the
Institute NumPor (KAUST) where part of this work was completed.

\end{document}